\documentclass[11pt,reqno]{amsart}
\usepackage{latexsym,amsfonts,amssymb,amsthm,amsmath,mathrsfs,color,amscd,graphicx,fullpage,parskip,hyperref,bm,bbm}

\usepackage{comment}
\usepackage{commath,mathtools,setspace}
\usepackage{paralist}
\usepackage{extarrows}

\excludecomment{comment} 

\newcommand{\s}[1]{{\mathcal #1}}

\newcommand{\bb}[1]{{\mathbb #1}}
\newcommand{\ip}[2]{\left\langle #1,#2 \right\rangle}
\newcommand{\sint}[1]{\int_{\mathbb{T}^d} #1}
\newcommand{\R}{\mathbb{R}}

\newtheorem{theorem}{Theorem} 
\newtheorem{corollary}[theorem]{Corollary}

\newtheorem{lemma}[theorem]{Lemma}
\newtheorem{proposition}[theorem]{Proposition}

\newtheorem{definition}[theorem]{Definition}

\newtheorem{remark}[theorem]{Remark}
\newtheorem{assumption}[theorem]{Assumption}

\numberwithin{equation}{section}
\numberwithin{theorem}{section}

\begin{document}

	\title[MFG of Controls]
	{Weak Solutions for Potential Mean Field Games of Controls}

	\author{P. Jameson Graber}
	\address{J.\@ Graber: Baylor University, Department of Mathematics;
		One Bear Place \#97328;
		Waco, TX 76798-7328 \\
		Tel.: +1-254-710- \\
		Fax: +1-254-710-3569 
	}
	\email{Jameson\_Graber@baylor.edu}
	
	\author{Alan Mullenix}
	\address{A.\@ Mullenix: Baylor University, Department of Mathematics;
		One Bear Place \#97328;
		Waco, TX 76798-7328 \\
		Tel.: +1-254-710- \\
		Fax: +1-254-710-3569 
	}
	\email{Alan\_Mullenix@baylor.edu}
	
	\author{Laurent Pfeiffer}
	\address{L.\@ Pfeiffer: Inria  and  CMAP  (UMR  7641),  CNRS,  Ecole  Polytechnique,  Institut  Polytechnique de  Paris, Route de Saclay, 91128 Palaiseau, France}
	\email{laurent.pfeiffer@inria.fr}
	
	\date{\today}   
	
	\begin{abstract}
	We analyze a system of partial differential equations that model a potential mean field game of controls, briefly MFGC. Such a game describes the interaction of infinitely many negligible players competing to optimize a personal value function that depends in aggregate on the state and, most notably, control choice of all other players. A solution of the system corresponds to a Nash Equilibrium, a group optimal strategy for which no one player can improve by altering only their own action. We investigate the second order, possibly degenerate, case with non-strictly elliptic diffusion operator and local coupling function. The main result exploits potentiality to employ variational techniques to provide a unique weak solution to the system, with additional space and time regularity results under additional assumptions. New analytical subtleties occur in obtaining a priori estimates with the introduction of an additional coupling that depends on the state distribution as well as feedback.
	\end{abstract}
	
	
	\maketitle

	\section{Introduction} \label{sec:intro}

	Mean Field Games (MFG), introduced simultaneously in 2006-7 by J.-M.\@ Lasry, P.-L.\@ Lions \cite{lasry07} and M.\@ Huang, R.\@ Malham{\'e}, P.\@ Caines \cite{huang2006large}, have seen swift development into a vibrant and substantial subfield of partial differential equations.
	See, for instance, the monographs \cite{carmona2017probabilistic,carmona2017probabilisticII,bensoussan2013mean}.
	 Considered are high population games of homogeneous, negligibly powerful players all attempting to optimize a cost while contending with the effects of the choices of all other players.

	\begin{comment}
	{\color{red} \textbf{Laurent: I would remove the text in red and rather give a heuristic interpretation later on.} The term Mean Field, inspired by physics, relates to each player viewing the remaining players as one large entity, seen in the interaction of \( f \) below, with a dependency on the distribution of player states, \( m(t,x) \).  Each player thus attempts to find, out of admissible strategies \( S \),
	\[ \inf_{S} \left\{  \int_t^T \underbrace{L(u(\tau),S(\tau))+f(u(\tau),m(\tau,u(\tau))}_{\text{Running Cost}} \dif \tau + \underbrace{g(u(T))}_{\text{Final Cost}} \right\} := \phi(t,x) \]
	Where \( \phi \) is called the value function for a particular player (the homogeneity giving all players the same value function). A Nash Equilibrium strategy \( s^* \) is found when, with all players employing \( s^* \), no single player can do better by a change in strategy.
	}
	\end{comment}
	
The term Mean Field, inspired by physics, relates to each player viewing the remaining players as one large entity. The cost functional that has to be optimized by each player typically incorporates an interaction term $f(m)$, where	 $m$ denotes the distribution of player states.
	Mean Field Games of Controls (briefly, MFGC), also called Extended Mean Field Games, introduce a control element into the Mean Field, so that not only can players ``detect" (via the Mean Field) the positions of others, but also their control choices.
	Such an extension naturally arises in many applications, for example in economics \cite{gueant2011mean,chan2015bertrand,chan2017fracking,graber2018existence,graber2018variational,graber2020commodities,graber2020nonlocal}.
	MFGC have been studied by D.\@ Gomes and V.\@ Voskanyan, who have results on classical solutions with S. Patrizi in the stationary (time independent) second order case where the diffusion is explicitly the Laplacian \cite{gomes2014extended}, and also in the time dependent first order case \cite{gomes2016extended}.
	In the second order uniformly parabolic time dependent case, Z.~Kobeissi has proved the existence of classical solutions under sufficient structural and smoothness assumptions, with uniqueness under additional assumptions, as well as results on approximate solutions \cite{kobeissi2019classical,kobeissi2020mean}.
	 P.~Cardaliaguet and C.-A.~Lehalle have provided a theorem giving the existence of weak solutions to a general system of MFGC, under the assumption that the Lagrangian is monotone with respect to the measure variable and that the Hamiltonian is sufficiently smooth; in particular it must depend on the density of players nonlocally \cite{cardaliaguet2017mfgcontrols}.
	 
	 In this article, we investigate the second order degenerate case (which can, in particular, be first order) featuring a non-strictly elliptic diffusion operator with space dependence. The MFGC system to be studied is
	\begin{equation}
	\label{eq:MFGC}
	\left\{
	\begin{array}{cll}
	(i)& -\partial_t u - A_{ij}\partial_{ij} u + H\del{x, D u(x,t) + \phi(x)^\intercal P(t)} = f\del{x,m(x,t)}
	& (x,t) \in Q,\\[0.3em]
	(ii)& \partial_t m - \partial_{ij}\del{A_{ij}m} + \nabla \cdot \del{vm} = 0
	& (x,t) \in Q,\\[0.3em]
	(iii)& P(t) = \Psi \big( \int_{\bb{T}^d} \phi(x)v(x,t)m(x,t) \dif x \big)
	& t \in [0,T],\\[0.3em]
	(iv)& v(x,t) = -D_\xi H\del{x,D u(x,t) + \phi(x)^\intercal P(t)}
	&(x,t) \in Q,\\[0.3em]
	(v)& m(x,0) = m_0(x), \quad u(x,T) = u_T(x),&
	x \in \bb{T}^d
	\end{array}\right.
	\end{equation}
	where $u,m$ are scalar functions, $v$ is a vector field in $\bb{R}^d$, $P = P(t) \in \bb{R}^k$, $Q := \bb{T}^d \times [0,T]$, and $A = [A_{ij}]_{1\leq i,j \leq d}$ is a given matrix-valued function on $\bb{T}^d$ whose values are symmetric and non-negative.

The heuristic interpretation of the above system is the following. Each agent controls the following dynamical system in $\bb{T}^d$:
\begin{equation*}
\dif X_t= \alpha_t \dif t + \sqrt{2} \Sigma(X_t) \dif B_t
\end{equation*}
where $(B_t)_{t \in [0,T]}$ is a standard Brownian motion in $\R^D$, $\alpha$ is an adapted process in $\R^d$, and $\Sigma \colon \bb{T}^d \rightarrow \R^{d \times D}$ is such that $A(x)= \Sigma(x) \Sigma(x)^\intercal$, for all $x \in \bb{T}^d$. The associated cost (to be minimized) is given by
\begin{equation*}
\mathbb{E} \Big[
\int_0^T \Big( H^*(X_t,-\alpha_t) + \langle P(t), \phi(X_t) \alpha_t \rangle + f(X_t,m(X_t,t)) \Big) \mathrm{d} t + u_T(X_T)
\Big].
\end{equation*}
At optimality, the control $\alpha$ is in feedback form, i.e.
\begin{equation*}
\alpha_t= v(X_t,t)=- D_p H(X_t, D u(X_t,t) + \phi(X_t)^\intercal P(t)).
\end{equation*}
On top of the classical interaction term $f(X_t,m(X_t,t))$, the price $P$ induces an interaction through the controls of the agents, since in equation $(iv)$, $P$ depends not only on $m$ but also on the feedback $v$.
	
	The basic structural assumptions are
	\begin{enumerate}
		\item $\bb{T}^d \times \bb{R}^d \ni (x,\xi) \mapsto H(x,\xi) \in \bb{R}$ is convex in $\xi$
		\item $\bb{T}^d \times [0,\infty) \ni (x,m) \mapsto f(x,m) \in \bb{R}$ is monotone increasing in $m$
		\item $\bb{R}^k \ni z \mapsto \Psi(z) \in \bb{R}$ is monotone in $z$, i.e.~$\langle \Psi(t,z_1)-\Psi(t,z_2), z_1-z_2 \rangle \geq 0$ for all $z_1,z_2 \in \bb{R}^k$.
	\end{enumerate}
See Subsection \ref{sec:asms} for more detailed assumptions on the data.

We will focus in the article on the MFG system obtained after performing the Benamou-Brenier change of variables $w = mv$ \cite{benamou2000computational}:
\begin{equation}
\label{eq:MFGC_bis}
\left\{
\begin{array}{cll}
(i)& -\partial_t u - A_{ij}\partial_{ij} u + H\del{x,D u(x,t) + \phi(x)^\intercal P(t)} = f\del{x,m(x,t)}
& (x,t) \in Q,\\[0.3em]
(ii)& \partial_t m - \partial_{ij}\del{A_{ij}m} + \nabla \cdot w = 0
& (x,t) \in Q,\\[0.3em]
(iii)& P(t) = \Psi \big( \int_{\bb{T}^d} \phi(x)w(x,t) \dif x \big)
& t \in [0,T],\\[0.3em]
(iv)& w(x,t) = - D_\xi H\del{x,D u(x,t) + \phi(x)^\intercal P(t)} m(x,t)
&(x,t) \in Q,\\[0.3em]
(v)& m(x,0) = m_0(x), \quad u(x,T) = u_T(x),&
x \in \bb{T}^d.
\end{array}\right.
\end{equation}
	
	In \cite{bonnans2019schauder} the authors prove the existence of classical solutions to \eqref{eq:MFGC} when $A$ is the identity matrix and the congestion term $f$ is nonlocal.
	In what follows we will provide the existence and uniqueness of a suitably defined ``weak solution" to the MFGC system with local coupling and provide additional regularity results involving the solution \( u \) and the distribution evolution \( m \). The method used follows the outline of Cardaliaguet, Graber, Porretta, and Tonon in \cite{cardaliaguet2015second}--see also  \cite{cardaliaguet2015weak,cardaliaguet2014mean}--and their treatment of the case of first and second order ``classical'' MFG systems. 
	The nonlocal interaction term $P(t)$ introduces new subtleties into the analysis, especially as it does not introduce any a priori gain of regularity.
	On the contrary, a priori estimates on solutions to the Hamilton-Jacobi Equation \eqref{eq:MFGC_bis}(i) are highly sensitive to the $L^p$ norms of $P(t)$.
	See Section \ref{sec:opt ctrl hj}.
	
	We first lay out the required assumptions on the data (Section \ref{sec:intro}). We then view the MFGC system as a system of optimality for two minimization problems (\( \inf D, \inf B \)) in duality (Section \ref{sec:duality});
	the Fenchel-Rockefellar duality theorem supplies uniqueness and some regularity for \( m,w \) and implies
		\[\inf_{\mathcal{K}_0} D(u,P,\gamma) = -\min_{\mathcal{K}_1} B(m,w). \]		
	Next, we show that the correct relaxation of \( \mathcal{K}_0 \) provides existence and a.e. uniqueness of a solution for the left hand side (Section \ref{sec:opt ctrl hj}). The solutions to these minimization problems are then shown to be proper candidates for the weak solution to the MFGC, whose existence is then proved (Section \ref{sec:existence}). Finally, with some additional assumptions on the data, we include some space and time regularity results for the weak solution based on previous techniques of Graber and Me{\'s}z{\'a}ros \cite{graber2018sobolev} (Section \ref{sec:regularity}).

We now lay out the notation and assumptions to hold throughout the paper.

\subsection{Notation}

We denote by 
$\ip{x}{y}$ the Euclidean scalar product  of two vectors $x,y\in\bb{R}^d$ and by
$|x|$ the Euclidean norm of $x$. We use conventions on repeated indices: for instance, if $a,b\in \bb{R}^d$, we often write $a_ib_i$ for the scalar product $\ip{a}{b}$. More generally, if $A$ and $B$ are two square symmetric matrices of size $d\times d$, we write $A_{ij}B_{ij}$ for ${\rm Tr}(AB)$.  
	
To avoid further difficulties arising from boundary issues, we work in the flat $d-$dimen\-sional torus $\bb{T}^d=\bb{R}^d\backslash \bb{Z}^d$. We denote by $P(\bb{T}^d)$ the set of Borel probability measures over $\bb{T}^d$. It is endowed with the weak convergence. For $k,n\in\bb{N}$ and $T>0$, we denote by ${\mathcal C}^k(Q, \bb{R}^n)$ the space of maps $\phi=\phi(t,x)$ of class ${\mathcal C}^k$ in time and space with values in $\bb{R}^n$. For $p\in [1,\infty]$ and $T>0$, we denote by $L^p(\bb{T}^d)$ and $L^p(Q)$ the set of $p-$integrable maps over $\bb{T}^d$ and $Q$ respectively. We often abbreviate $L^p(\bb{T}^d)$ and $L^p(Q)$ into  $L^p$. We denote by $\|f\|_p$ the $L^p-$norm of a map $f\in L^p$. The conjugate of a real $p>1$ is denoted by $p'$, i.e.\@ $1/p + 1/p'= 1$.
	
\subsection{Assumptions} \label{sec:asms}
    
We now collect the assumptions on the ``congestion coupling" $f$, the ``aggregate control coupling" $\Psi$, the Hamiltonian $H$, and the initial and terminal conditions $m_0$ and $u_T$.

Along the article, we assume that there exist some constants $C_1 > 0$, $C_2 > 0$, $C_3 > 0$, $C_4 > 0$, $q> 1$, $r> 1$, and $s>1$ such that the following hypotheses hold true. We denote
\begin{equation*}
p= q'.
\end{equation*}

\begin{itemize}
\item[(H1)] (Conditions on the coupling)
\begin{itemize}
\item[$\bullet \ \ $] The map $f \colon \bb{T}^d\times [0,+\infty)\to \bb{R}$ is continuous in both variables, increasing with respect to the second variable $m$, and satisfies
\begin{equation}
\label{Hypf}
\frac{1}{C_1}|m|^{q-1}-C_1\leq f(x,m) \leq C_1 |m|^{q-1}+C_1 \qquad \forall m\geq 0 \;.
\end{equation}
Moreover $f(x,0)= 0$ for all $x \in \bb{T}^d$.
\item[$\bullet \ \ $] The map $\Psi \colon \bb{R}^k \rightarrow \bb{R}^k$ is the continuous gradient of some convex function $\Phi \colon \bb{R}^k \rightarrow \R$. Without loss of generality, we assume that $\Phi(0)= 0$. Moreover,
\begin{equation} \label{eq:growth_Phi}
\Phi(z) \leq C_2 |z|^s + C_2 \qquad \forall z \in \R^k.
\end{equation}
Changing $C_2$ if necessary, we have
\begin{equation*}
\Phi^*(z) \geq \frac{1}{C_2} |z|^{s'} - C_2 \qquad  \forall z \in \R^k.
\end{equation*}
If $\frac{1}{s} + \frac{1}{pr} < 1$, we assume that
\begin{equation} \label{eq:low_bound_phi}
\frac{1}{C_2} |z|^s - C_2 \leq \Phi(z)\qquad \forall z \in \R^k .
\end{equation}
\item[$\bullet \ \ $] The map $\phi \colon \bb{T}^d \to \s{L}(\bb{R}^d;\bb{R}^k)$ is continuously differentiable.
If $\frac{1}{s} + \frac{1}{pr} < 1$, we assume that it is constant.
\end{itemize}

\item[(H2)] (Conditions on the Hamiltonian) The Hamiltonian  $H \colon \bb{T}^d\times \bb{R}^d\to\bb{R}$ is continuous in both variables, convex and differentiable in the second variable, with $D_{\xi} H$ continuous in both variables, and has a superlinear growth in the gradient variable:
\begin{equation}\label{HypGrowthH}
\frac{1}{C_3} |\xi|^{r} - C_3 \leq H(x,\xi) \leq C_3|\xi|^r + C_3 \qquad \forall (x,\xi)\in \bb{T}^d\times \bb{R}^d.
\end{equation}
We note for later use that the Fenchel conjugate $H^*$ of $H$ with respect to the second variable is continuous and satisfies similar inequalities (changing $C_3$ if necessary):
\begin{equation}\label{HypHstar}
\frac{1}{C_3} |\xi|^{r'} -C_3 \leq H^*(x,\xi) \leq C_3|\xi|^{r'} + C_3 \qquad \forall (x,\xi)\in \bb{T}^d\times \bb{R}^d.
\end{equation}

\item[(H3)] (Conditions on $A$) There exists a Lipschitz continuous map $\Sigma:\bb{T}^d\to \bb{R}^{d\times D}$  such that $\Sigma\Sigma^T=A$ and such that
\begin{equation}\label{HypA}
|\Sigma(x)-\Sigma(y)|\leq C_4 |x-y|  \qquad \forall x,y\in \bb{T}^d.
\end{equation}

\item[(H4)] (Conditions on the initial and terminal conditions) $\phi_T:\bb{T}^d\to \bb{R}$ is of class ${\mathcal C}^2$, while $m_0:\bb{T}^d\to \bb{R}$ is a $C^1$ positive density (namely $m_0>0$ and $\int_{\bb{T}^d} m_0 \dif x=1$). 

\item[(H5)] (Restrictions on the exponents). We consider 4 cases, depending on whether $s' < r$ or $s' \geq r$ and whether $A$ is constant or not.
\begin{equation*}
\begin{array}{|c||c|c|}
\hline & & \\[-4mm]
 & \text{Case 1: } s'<r & \text{Case 2: } s' \geq r \\ \hline \hline & & \\[-4mm]
\begin{array}{c} \text{Case A:} \\ \! \! \! \! \text{$A$ is not constant} \! \! \! \! \end{array} & s' \geq p & r \geq p \\ \hline & & \\[-3mm]
\begin{array}{c} \text{Case B:} \\ \text{$A$ is constant} \end{array} & {\displaystyle \frac{s'(d+1)}{d} \geq p } &
\begin{array}{c} \! \! \big[ s' \geq 1+d \big] \quad \text{or} \ \\[1mm] \big[ s' < 1+d \ \text{ and } \ \frac{s'(1+d)}{d-s'+1} > p \big] \end{array} \\[5mm] \hline
\end{array}
\end{equation*}
\end{itemize}

\begin{remark} 
\begin{enumerate}
\item The condition $f(x,0)= 0$ is just a normalization condition, which we may assume without loss of generality, as explained in \cite[Section 2]{cardaliaguet2015second}.
\item
Let us compare the different cases of Assumption (H5).
\begin{enumerate}
\item Assumption (H5) is stronger in cases 1A and 2A than in cases 1B and 2B, respectively, that is, Assumption (H5) is stronger in the case of a non-constant $A$ than in the constant case.
\item If $A$ is not constant (cases 1A and 2A), than Assumption (H5) can be summarized by $\min(s',r) \geq p$.
\item If $A$ is constant, it is easy to verify that Assumption (H5) is stronger in the case 1B $(s'<r)$ than in case 2B $(s' \geq r)$.
\end{enumerate}
\item If $\Psi=0$, then we are back to the framework of \cite{cardaliaguet2015second} and our assumptions coincide. Indeed, \eqref{eq:growth_Phi} is then satisfied with any $s > 1$. Taking $s$ sufficiently close to 1, we have $1/s + 1/(rp) \geq 1$, so that \eqref{eq:low_bound_phi} is not necessary, and we have $s' \geq r$, so that we are either in case 2A or 2B in hypothesis (H5). If $A$ is constant, we must choose $s$ close enough to 1, so that $s' \geq 1+d$.
\end{enumerate}
\end{remark}
	
Let us set 
\begin{equation*}
F(x,m)=
\begin{cases} \begin{array}{ll}
{\displaystyle \int_0^m f(x,\tau)d\tau} & {\rm if }\; m\geq 0\\[1em]
+\infty & {\rm otherwise}.
\end{array}
\end{cases}
\end{equation*}
Then $F$ is continuous on $\bb{T}^d\times (0,+\infty)$, differentiable and strictly convex in $m$ and satisfies
\begin{equation}\label{HypGrowthF}
\frac{1}{C_1}|m|^{q} - C_1 \leq F(x,m) \leq  {C_1}|m|^{q}+C_1 \qquad \forall m\geq 0,
\end{equation}
changing $C_1$ if necessary.
Let $F^*$ be the Fenchel conjugate of $F$ with respect to the second variable. Note that $F^*(x,a)=0$ for $a\leq 0$ because $F(x,m)$ is nonnegative, equal to $+\infty$ for $m<0$, and null at zero.  Moreover, 
\begin{equation}\label{HypGrowthFstar}
\frac{1}{C_1}|a|^{p}-C_1\leq F^*(x,a) \leq  C_1 |a|^{p} + C_1 \qquad \forall a \geq 0,
\end{equation}
changing again $C_1$ if necessary.

\begin{remark}
	Most of the results in this paper hold also for time-dependent data, in particular when $f$ and $H$ depend on $t$.
	It suffices to have the estimates in this subsection hold uniformly with respect to $t$.
\end{remark}
	
\section{Two problems in duality} \label{sec:duality}

The approach that we follow consists in viewing the MFG system as an optimality condition for two convex problems, which we introduce now.

Let $\s{K}_0$ be the set of all triples $(u,P,\gamma) \in \mathcal{C}^2(Q) \times \mathcal{C}^0([0,T];\R^k) \times \mathcal{C}^0(Q)$ satisfying
\begin{equation}
\label{eq:HJ subsolution}
\left\{
\begin{array}{rcl}
-\partial_t u - A_{ij}\partial_{ij} u + H\del{x,D u(x,t) + \phi(x)^\intercal P(t)}
& = & \gamma,\\
u(x,T) &=& u_T(x).
\end{array}\right.
\end{equation}
The associated cost is given by
\begin{equation}
D(u,P,\gamma) = -\int_{\bb{T}^d} u(x,0)m_0(x)\dif x + \int_0^T \Phi^*\del{P(t)}\dif t + \iint_Q F^*\del{x,\gamma(x,t)}\dif x\dif t.
\end{equation}
The first problem is
\begin{equation} \label{eq:hjb_problem}
\inf_{(u,P,\gamma) \in \mathcal{K}_0} D(u,P,\gamma).
\end{equation}

We consider now the set $\s{K}_1$ of all pairs $(m,w) \in L^1(Q) \times L^1(Q;\R^d)$ such that $m \geq 0$ a.e., $\int_{\bb{T}^d} m(x,t) \dif x=1$ for a.e.\@ $t \in (0,T)$,
and such that the continuity equation
\begin{equation}
\label{eq:FP equation}
\left\{
\begin{array}{rcl}
\partial_t m - \partial_{ij}\del{A_{ij}m} + \nabla \cdot w
& = & 0,\\
m(x,0) &=& m_0(x)
\end{array}\right.
\end{equation}
holds in the sense of distributions.
For $(m,w) \in \s{K}_1$, consider the following cost functional:
\begin{equation}
\begin{array}{rl}
B(m,w) = & {\displaystyle \iint_Q \del{H^*\del{x,-\frac{w(x,t)}{m(x,t)}}m(x,t) + F\del{x,m(x,t)}} \dif x \dif t} \\[1.5em]
& \qquad {\displaystyle + \int_0^T \Phi\del{\int_{\bb{T}^d} \phi(x)w(x,t) \dif x}\dif t + \int_{\bb{T}^d} u_T(x)m(x,T)\dif x},
\end{array}
\end{equation}
where for $m(t,x)=0$, we impose that
\begin{equation*}
m(t,x) H^* \Big(x,-\frac{w(t,x)}{m(t,x)} \Big) =
\begin{cases}
\begin{array}{cl}
+ \infty & \text{if $w(t,x) \neq 0$} \\
0 & \text{if $w(t,x)= 0$}.
\end{array}
\end{cases}
\end{equation*}
Since $H^*$ and $F$ are bounded from below and $m \geq 0$, the first integral in $B$ is well defined in $\R \cup \{ + \infty \}$.
Concerning the second term in $B$, we simply need to observe that since $\Phi$ is convex, there exists a constant $C>0$ such that $\Phi(z) \geq -C(1 + |z|)$, for all $z \in \R^k$. For $w \in L^1(Q;\R^d)$, the term
\begin{equation*}
\int_{\bb{T}^d} \phi(x) w(x,\cdot) \dif x
\end{equation*}
is integrable in time and therefore
\begin{equation*}
\int_0^T \Phi\del{\int_{\bb{T}^d} \phi(x)w(x,t) \dif x}\dif t
\end{equation*}
is well-defined in $\R \cup \{ + \infty \}$.
For the third term, we refer the reader to \cite[Lemma 4.1]{cardaliaguet2015second}, where it is proved that for $(m,w) \in \s{K}_1$, $m$ can be seen as a continuous map from $[0,T]$ to $P(\bb{T}^d)$ for the Rubinstein-Kantorovich distance $\mathbf{d}_1$.
Finally, the second optimization problem is the following:
\begin{equation} \label{eq:fp_problem}
\inf_{(m,w) \in \mathcal{K}_1} B(m,w).
\end{equation}


\begin{lemma} \label{lemma:duality}
We have
\begin{equation*}
\inf_{(u,P,\gamma) \in \s{K}_0} D(u,P,\gamma) = -\min_{(m,w) \in \s{K}_1} B(m,w).
\end{equation*}
Moreover, the minimum in the right-hand side is achieved by a unique pair \( (m,w) \in \s{K}_1 \) satisfying
\begin{equation} \label{eq:bound_mw}
m \in L^q(Q), \quad
w \in L^{\frac{r'q}{r'+q-1}}(Q;\R^d), \quad \text{and} \quad
\int_{\bb{T}^d} \phi(x) w(x,\cdot) \dif x \in L^s((0,T);\R^k).
\end{equation}
\end{lemma}

\begin{proof}
Following previous papers \cite{cardaliaguet2015weak,cardaliaguet2014mean,cardaliaguet2015second}, we look to apply the Fenchel-Rockafellar duality theorem. In order to do so, we reformulate the first optimization problem into a more suitable form. 
    
Let $E_0 = \s{C}^2(Q) \times \s{C}^0([0,T], \mathbb{R}^k)$ and $E_1 = \s{C}^0(Q) \times \s{C}^0(Q; \mathbb{R}^d)$. 
Define on \( E_0 \) the functional
\begin{equation*}
\mathcal{F}(u,P) = \int_0^T \Phi^*(P(t)) \dif t  -\int_{\mathbb{T}^d} u(0,x) m_0(x)  \ \dif x + \chi_S(u), \end{equation*}
where \( \chi_S \) is the convex characteristic function of the set \( S = \left\{ u \in E_0, u(T,\cdot)=u_T \right\} \), i.e., \( \chi_S(u) = 0 \) if \( u \in S \) and \( + \infty \) otherwise. For \( (a,b) \in E_1 \), we define
\begin{equation*}
\mathcal{G}(a,b) = \iint_Q F^*\del{x,-a+H(x,b)} \dif x \dif t.
\end{equation*}
The functional \( \mathcal{F} \) is convex and lower semi-continuous on \( E_0 \) while \( \mathcal{G} \) is convex and continuous on \( E_1 \). Let \( \Lambda : E_0 \to E_1 \) be the bounded linear operator defined by \( \Lambda( u, P) = ( \partial_t u + A_{ij} \partial_{ij} u, D u + \phi^\intercal P) \). We can observe that
\begin{equation*}
\inf_{ (u,P,\gamma) \in \mathcal{K}_0 } D(u,P,\gamma) = \inf_{(u,P) \in E_0} \left\{ \mathcal{F}(u,P)+\mathcal{G}(\Lambda(u,P)) \right\}.
\end{equation*}
In the interest of employing the Fenchel-Rockafellar duality theorem, note that \( \s{F}(u,P) < +\infty \) for \( (u_T, 0) \) and \( \s{G}(\Lambda(u,P)) \) is continuous at \( ( u_T, 0 ) \). This satisfies the duality theorem, and so 
\[ \inf_{(u,P) \in E_0} \left\{ \mathcal{F}(u,P)+\mathcal{G}(\Lambda(u,P)) \right\} = \max_{(m,w) \in E'_1} \left\{ -\mathcal{F}^*(\Lambda^*(m,w))-\mathcal{G}^*(-(m,w)) \right\}, \]
where \( E'_1 \) is the dual space of \( E_1 \), i.e., the set of vector valued Radon measures \( (m,w) \) over \( Q \) with values in \( \mathbb{R} \times \mathbb{R}^d \), \( E'_0 \) is the dual space of \( E_0 \), \( \Lambda^* : E'_1 \to E'_0 \) is the dual operator of \( \Lambda \) and \( \s{F}^* \) and \( \s{G}^* \) are the convex conjugates of \( \s{F} \) and \( \s{G} \) respectively. We now compute the relevant conjugate transforms.
\begin{align*}
& \mathcal{F}^*(\Lambda^*(m,w))
= \sup_{(u,P)} \Big\{ \langle (u,P), \Lambda^*(m,w) \rangle - \mathcal{F}(u,P) \Big\} \\
& \quad = \sup_{\substack{(u,P) \\ u \in S}} \Big\{ \langle \Lambda(u,P), (m,w) \rangle - \int_0^T \Phi^*(P(t)) \dif t + \int_{\mathbb{T}^d} u(0,x) m_0(x) \dif x \Big\} \\
& \quad = \sup_{\substack{(u,P) \\ u \in S}} \Big\{  \langle \partial_t u + A_{ij}\partial_{ij} u, m \rangle  + \langle D u, w \rangle + \langle \phi^\intercal P, w \rangle
 - \int_0^T \Phi^*(P(t)) \dif t + \int_{\mathbb{T}^d} u(0,x)m_0(x) \dif x \Big\} \\
& \quad = \sup_{\substack{(u,P) \\ u \in S}} \Big\{ \langle - \partial_t m + \partial_{ij}(A_{ij}m) - \nabla \cdot w , u  \rangle  +  \langle  \phi^\intercal P , w \rangle - \int_0^T \Phi^*(P(t)) \dif t + \int_{\mathbb{T}^d} u_T(x) m(T,x) \dif x \Big\}.
\end{align*}
%
It is evident here that if \( -\partial_t m + \partial_{ij}(A_{ij}  m) - \nabla \cdot w \neq 0 \) in the sense of distributions, this supremum is infinite. If the condition does hold however, the supremum no longer depends on \( u \), and so the calculation is reduced to
\begin{align*}
& \sup_{P} \, \Big\{ \int_0^T \Big( {\textstyle \int_{\bb{T}^d} } \langle \phi(x)^\intercal P(t), w(x,t) \rangle \dif x \Big) - \Phi^*(P(t))  \dif t \Big\} \\
& \qquad =
\sup_{P} \, \Big\{ \int_0^T \Big\langle P(t), {\textstyle \int_{\bb{T}^d} } \phi(x) w(x,t) \dif x \Big\rangle  - \Phi^*(P(t))  \dif t \Big\} \\
& \qquad = \sup_{P} \, \Big\{ \int_0^T \Phi \big( {\textstyle \int_{\mathbb{T}^d} } \phi(x) w(x,t) \dif x \big) \dif t \Big\}.
\end{align*}
Combined with the conditions above, we have
\begin{equation*}
\mathcal{F}^*(\Lambda^*(m,w)) = \begin{cases} {\displaystyle
\int_0^T \Phi \big( {\textstyle \int_{\mathbb{T}^d} } \phi w \big) + \int_{\mathbb{T}^d} u_T(x) m(T,x) \dif x, } & \text{ if $-\partial_t m + A_{ij} \partial_{ij} m - \nabla \cdot w = 0$,} \\[1em]
+\infty  &  \text{ otherwise.}
\end{cases}
\end{equation*}
Following \cite{cardaliaguet2015weak}, we have that $\mathcal{G}^*(m,w)= + \infty$ if $(m,w) \notin L^1(Q) \times L^1(Q;\R^d)$ and
\begin{equation*}
\mathcal{G}^*(m,w) = \iint_Q K^*(x,m(t,x),w(t,x)) \dif x \dif t
\end{equation*}
otherwise,
where \( K^* \) is given by 
\begin{equation*}
K^*(x,m,w) = \begin{cases}
 F(x,-m) - m H^* \Big( x, -\frac{w}{m} \Big) & \text{ if } m < 0 \\
 0 & \text{ if } m=0 \text{ and } w=0 \\
 +\infty & \text{ otherwise},
\end{cases}
\end{equation*}
it is the convex conjugate of 
\begin{equation*}
K(x,a,b) = F^*(x,-a+H(x,b)) \qquad \forall(x,a,b) \in \mathbb{T}^d \times \mathbb{R} \times \mathbb{R}^d.
\end{equation*}
Therefore
\begin{align*}
& \max_{(m,w) \in E'_1} \big\{ - \s{F}^*(\Lambda^*(m,w)) - \s{G}^*(-(m,w)) \big\} \\
& \qquad = \max \Big\{ \iint_Q -F(x,m) - mH^* \Big( x,-\frac{w}{m} \Big) \dif x \dif t - \int_0^T \Phi \left( {\textstyle \int_{\mathbb{T}^d} } \phi w \right) \dif t - \int_{\mathbb{T}^d} u_T(x) m(T,x) \dif x \Big\}
\end{align*}
with the last maximum taken over the \( L^1 \) maps \( ( m,w ) \) such that \( m \geq 0 \) a.e.\@ and
\begin{equation*}
-\partial_t m + \partial_{ij}(A_{ij}  m) - \nabla \cdot w = 0, \quad m(0) = m_0
\end{equation*}
holds in the sense of distributions. Since \( \int_{\mathbb{T}^d} m_0 = 1 \), it follows that \( \int_{\mathbb{T}^d} m(t) = 1 \) for any \( t \in [0,T] \).
Thus the pair \( (m,w) \) belongs to the set \( \s{K}_1 \) and the first part of the statement is proved.

It remains to show \eqref{eq:bound_mw}.
Taking an optimal \( (m,w) \in \s{K}_1 \) in the above problem, we have that \( w(t,x) = 0 \) for all \( (t,x) \in [0,T] \times \mathbb{T}^d \) whenever \( m(t,x) = 0 \). 
By convexity of $\Phi$, we have
\begin{equation} \label{eq:lower_bound_phi}
\int_0^T \Phi \big( {\textstyle \int_{\bb{T}^d} } \phi w \dif x \big) \dif t
\geq \int_0^T \Phi(0) + \big\langle \Psi(0),{\textstyle \int_{\bb{T}^d} \phi w \dif x } \big\rangle \dif t
\geq C - C \| w \|_1.
\end{equation}
Moreover, by H\"older's inequality,
\begin{align*}
\| w \|_1
= \iint_Q \frac{|w|}{m} m \dif x \dif t
\leq \Big( \iint_Q \Big( \frac{|w|}{m} \Big)^{r'} m \Big)^{1/r'}.
\end{align*}
It follows with Young's inequality that
\begin{align*}
\| w \|_1
\leq \frac{\varepsilon}{r'} \iint_Q |w|^{r'} m^{1-r'} \dif x \dif t + \varepsilon^{-(r-1)},
\end{align*}
for any $\varepsilon > 0$. Therefore,
\begin{equation*}
\int_0^T \Phi \big( {\textstyle \int_{\bb{T}^d} } \phi w \dif x \big) \dif t
\geq C - \frac{\varepsilon C}{r'} \iint_Q |w|^{r'} m^{1-r'} \dif x \dif t - \varepsilon^{-(r-1)}.
\end{equation*}
Using successively the optimality of $(m,w)$, the growth conditions on $F$ and $H^*$ and the above inequality, we obtain
\begin{align*}
C & \geq \iint_Q F(x,m) + m H^* \Big( x, -\frac{w}{m} \Big) \dif x \dif t  + \int_0^T \Phi \big( {\textstyle \int_{\mathbb{T}^d} } \phi w \big) \dif t + \int_{\mathbb{T}^d} u_T(x) m(T,x) \dif x \\
& \geq \frac{1}{C} \| m \|_q^q
+ \Big( \frac{1}{C} - \varepsilon C \Big) \iint_Q |w|^{r'} m^{1-r'} \dif x \dif t - \varepsilon^{-(r-1)}-C,
\end{align*}
for some constant $C > 0$ independent of $\varepsilon$. Choosing $\varepsilon > 0$ sufficiently small, we deduce that $m \in L^q(Q)$ and that $|w|^{r'} m^{1-r'} \in L^1(Q)$.
%
%
To investigate the claim on \( w \) given in the statement, write, for some \( \rho \) to be determined,
\begin{equation*}
\| w \|_{\rho r'}^{r'} = \enVert{ m^{r'-1} \frac{|w|^{r'}}{m^{r'-1}}}_\rho \leq  \enVert{m^{r'-1}}_{\frac{q}{r'-1}}  \enVert{\frac{|w|^{r'}}{m^{r'-1}}}_1 = \| m \|_{q}^{r'-1} \enVert{\frac{|w|^{r'}}{m^{r'-1}}}_1 < \infty.
\end{equation*}
For this interpolative H{\"o}lder inequality to hold, we must have \( \rho = \frac{q}{r'+q-1} \). Thus $w \in L^{\sigma}(Q; \R^d)$, with
\begin{equation}
\sigma= \frac{r'q}{r'+q-1}, \quad \text{i.e. }
\sigma'= rp.
\end{equation}
Two cases must be considered. If $\frac{1}{s} \geq 1- \frac{1}{rp} = \frac{1}{\sigma}$, then we have $\sigma \geq s$, thus $w \in L^s(Q;\R^d)$ and \eqref{eq:bound_mw} follows. In the other case, we have by Hypothesis (H1) the growth assumption $\Phi(z) \geq \frac{1}{C} |z|^s - C$. It can be employed to get a better bound from below in \eqref{eq:lower_bound_phi}. We obtain \eqref{eq:bound_mw} with a straightforward adaptation of the above proof.
\end{proof}

\section{Optimal control problem of the HJ equation} \label{sec:opt ctrl hj}

In general we do not expect problem \eqref{eq:hjb_problem} to have a solution. 
In this section we exhibit a relaxation for \eqref{eq:hjb_problem} (Proposition \ref{prop:justif_relaxation}) and show that the obtained relaxed problem has at least one solution (Proposition \ref{prop:existence_solution}).

\subsection{Estimates on subsolutions to HJ equations} \label{subsection:estimates_hj}

In this subsection we prove estimates in Lebesgue spaces for subsolutions of Hamilton-Jacobi equations of the form
\begin{equation} \label{eq:subsolution_HJ}
\begin{cases}
\begin{array}{cl}
(i) & - \partial_t u - A_{ij} \partial_{ij} u + H( Du + \phi^\intercal P) \leq \gamma \\
(ii) & u(x,T) \leq u_T(x)
\end{array}
\end{cases}
\end{equation}
in terms of Lebesgue norms of $\gamma$, $u_T$, and $P$. Equation \eqref{eq:subsolution_HJ} is understood in the sense of distributions. This means that $D u + \phi^\intercal P \in L^r(Q)$ and for any nonegative test function $\zeta \in C_c^\infty((0,T] \times \bb{T}^d)$,
\begin{equation} \label{eq:weak_hj_def}
- \int_{\bb{T}^d} \zeta(T) u_T
+ \iint_Q u \partial_t \zeta
+ \langle D \zeta, A D u \rangle
+ \zeta \big( \partial_i A_{ij} \partial_j u + H(D u + \phi^{\intercal} P) \big)
\leq
\iint_Q \gamma \zeta.
\end{equation}

Let us introduce some notation.
For all $\tilde{r} > 1$ and for all $\tilde{p} \geq 1$, let us define $\bar{\kappa}(\tilde{r},\tilde{p})$ and $\bar{\eta}(\tilde{r},\tilde{p})$ by
\begin{equation*}
\bar{\eta}(\tilde{r},\tilde{p})
= \frac{d(\tilde{r}(\tilde{p}-1)+1)}{d-\tilde{r}(\tilde{p}-1)}
\quad \text{and} \quad
\bar{\kappa}(\tilde{r},\tilde{p})
= \frac{\tilde{r} \tilde{p}(1+d)}{d-\tilde{r}(\tilde{p}-1)}
\end{equation*}
if $\tilde{p} < 1 + \frac{d}{\tilde{r}}$ and
\begin{equation*}
\bar{\eta}(\tilde{r},\tilde{p})= \infty \quad \text{and} \quad
\bar{\kappa}(\tilde{r},\tilde{p})= \infty
\end{equation*}
if $\tilde{p} > 1 + \frac{d}{\tilde{r}}$. In the border line case $\tilde{p}= 1 + \frac{d}{\tilde{r}}$, $\bar{\eta}(\tilde{r},\tilde{p})$ and $\bar{\kappa}(\tilde{r},\tilde{p})$ can be fixed to arbitrarily large values.
We let the reader verify that
\begin{equation} \label{eq:low_bound_kappa}
\bar{\kappa}(\tilde{r},\tilde{p}) \geq \tilde{r} \quad \text{and} \quad
\bar{\kappa}(\tilde{r},\tilde{p}) \geq \tilde{p},
\end{equation}
assuming that the assigned value to $\bar{\kappa}(\tilde{r},\tilde{p})$ is large enough in the border line case.

We now restate \cite[Theorem 3.3]{cardaliaguet2015second}, since it will prove useful below.
\begin{theorem} \label{theo:upper_bound_hj}
Let \( u \) satisfy
\begin{equation} \label{eq:subsolution_HJ-1}
\begin{cases}
\begin{array}{cl}
(i) & - \partial_t u - A_{ij} \partial_{ij} u + \frac{1}{K}\abs{D u}^{\tilde r} \leq \gamma \\
(ii) & u(x,T) \leq u_T(x)
\end{array}
\end{cases}
\end{equation}
in the sense of distributions, with \( \gamma \in L^{\tilde p}(Q) \) for some $\tilde p \geq 1$ and $\tilde r > 1$.
Then
\begin{equation*}
\| u_+ \|_{L^\infty((0,T),L^\eta(\bb{T}))}
+
\| u_+ \|_{L^\kappa(Q)}
\leq C,
\end{equation*}
where $u_+ := \max\{u,0\}$, $\kappa= \bar{\kappa}(\tilde{r},\tilde{p})$, $\eta= \bar{\eta}(\tilde{r},\tilde{p})$.
The constant \( C \) depends only on $T,d,\tilde r,\tilde p,$ $K$ (appearing in \eqref{eq:subsolution_HJ-1}), $C_4$ (appearing in Hypothesis (H3)), $\| \gamma \|_{\tilde{p}}$, and $\enVert{u_T}_{\eta}$.
\end{theorem}

\begin{remark}
Although the case $\tilde p = 1$ is not explicitly mentioned in \cite[Theorem 3.3]{cardaliaguet2015second}, it is not hard to check that the theorem also applies in that case.
\end{remark}

\begin{corollary} \label{cor:a priori bound1}
Let $P \in \mathcal{C}^0([0,T];\R^k)$ and $\gamma \in \mathcal{C}^0(Q;\R)$. Let $u$ be the viscosity solution to the HJ equation
\begin{equation} \label{eq:HJB_coro}
	\left\{
	\begin{array}{rcl}
	-\partial_t u - A_{ij}\partial_{ij} u + H\del{x,D u(x,t) + \phi(x,t)^\intercal P(t)}
	& = & \gamma,\\
	u(x,T) &=& u_T(x).
	\end{array}\right.
\end{equation}
Let $\tilde{r} > 1$. Define
	\begin{equation} \label{eq:kappa1}
	\kappa= \bar{\kappa}(\tilde{r},1)= \frac{\tilde r(1+d)}{d}.
	\end{equation}
Then 
	\begin{equation*}
	\| u \|_{L^\infty((0,T),L^1(\bb{T}))}
	+
	\| u \|_{L^\kappa(Q)}
	\leq C,
	\end{equation*}
	where the constant \( C \) depends only on $T,d,\tilde r,\tilde p$, $C_4$ (appearing in Hypothesis (H3)), $\| \gamma \|_{1}$, $\enVert{D u}_{\tilde r}$, $\| H(D u + \phi^\intercal P) \|_{1}$,  and $\enVert{u_T}_{1}$.
\end{corollary}

\begin{proof}
By \cite{ishii95}, $u$ also satisfies the HJ equation in the sense of distributions. Observe that \eqref{eq:HJB_coro} can be rewritten
	\begin{equation}
	- \partial_t u - A_{ij} \partial_{ij} u + \abs{D u}^{\tilde r} = \gamma - H(D u + \phi^\intercal P) + \abs{D u}^{\tilde r}.
	\end{equation}
The $L^1$ norm of the right-hand side depends on $\| \gamma \|_{1}$, $\| H(D u + \phi^\intercal P) \|_{1}$, and $\enVert{D u}_{\tilde r}$.
Similarly, $-u$ is a weak subsolution of a HJ equation with right-hand side
\begin{equation*}
-\gamma + H(D u + \phi^\intercal P) + | D u |^{\tilde{r}}.
\end{equation*}
By applying Theorem \ref{theo:upper_bound_hj} to both $u$ and $-u$, we deduce the desired estimate.
\end{proof}

When $s' \geq r$, the growth assumption on the Hamiltonian (Hypothesis (H2)) can be exploited to derive a more precise estimate on the solution to \eqref{eq:HJB_coro}. 

\begin{corollary} \label{cor:a priori bound2}
Let $P$, $u$, and $\gamma$ be as in Corollary \ref{cor:a priori bound1}. Assume moreover that $\gamma \geq 0$ and $s' \geq r$. Take $\tilde{r} \in (1,r]$ and define
\begin{equation*}
\tilde{p} = \min \Big( p,\frac{s'}{\tilde r} \Big), \quad
\kappa = \bar{\kappa}(\tilde{r},\tilde{p}), \quad \text{and} \quad
\eta= \bar{\eta}(\tilde{r},\tilde{p}).
\end{equation*}
Then
\begin{equation*}
\| u\|_{L^\infty((0,T),L^\eta(\bb{T}^d))}
+
\| u \|_{L^\kappa(Q)}
\leq C.
\end{equation*}
The constant \( C \) depends only on $T,d,\tilde r,\tilde p,$ $C_3$ (appearing in Hypothesis (H2)), $C_4$ (appearing in Hypothesis (H3)), $\| \gamma \|_{p}$, $\enVert{P}_{s'}$, and $\| u_T \|_{\eta}$.
\end{corollary}

\begin{proof}
	We have $\gamma \geq 0$ and the upper bound
	\begin{equation*}
	H(D u + \phi^{\intercal} P)
	\leq C | D u |^r + C | P |^r,
	\end{equation*}
	therefore,
	\begin{equation} \label{eq:existence4}
	-\partial_t u - A_{ij} \partial_{ij} u + C | D u |^r \geq  - C |P|^r - C, \quad
	u(T,x) \geq \Big( \min_{x' \in \bb{T}^d} u_T(x') \Big).
	\end{equation}
	Let $\hat{u}$ be defined by
	\begin{equation*}
	\hat{u}(x,t)= \Big( \min_{x' \in \bb{T}^d} u_T(x') \Big)
	- \int_t^T |P(t)|^r \dif t - C(T-t),
	\end{equation*}
	which solves \eqref{eq:existence4} with inequality replaced by equality.
	By the comparison principle, $u \geq \hat u \geq -C$, where $C$ depends only on $T,$ the growth of $H$, $\min u_T$, and $\enVert{P}_{r}$.
	Note that $\enVert{P}_{r}$ depends only on $\enVert{P}_{s'}$ and $T$ because $s' \geq r$.
	
	Next, by the growth condition of \( H \) we have
	\[
	- \partial_t u - A_{ij} \partial_{ij}u + \frac{1}{C_3}|D u + \phi^\intercal P|^r - C_3 \leq \gamma. 
	\]
Observe that by Young's inequality,
\begin{align*}
| D u |^{\tilde{r}}
\leq 2^{\tilde r -1}| D u + \phi^\intercal P|^{\tilde{r}} + 2^{\tilde r -1}|\phi^\intercal P|^{\tilde{r}}
\leq 2^{\tilde r -1} \Big( \frac{\tilde{r}}{r} | D u + \phi^\intercal P|^{r} + 1  \Big) + C |P|^{\tilde{r}},
\end{align*}
since $r \geq \tilde{r}$.
It follows that
\begin{equation*}
| D u + \phi^\intercal P|^{r}
\geq \frac{1}{C}  |D u|^{\tilde{r}} - C |P|^{\tilde{r}} - C,
\end{equation*}
therefore,
\begin{equation} \label{thm31_1}
-\partial_t u - A_{ij} \partial_{ij} u + \frac{1}{C} |D u|^{\tilde{r}} \leq \gamma + C|P|^{\tilde r} + C.
\end{equation}
Since $|P|^{\tilde{r}}$ lies in $L^{s'/\tilde{r}}(Q)$, we have that the right hand side of \eqref{thm31_1} is bounded in \( L^{\tilde{p}} \). Combining this with the lower bound on $u$, the conclusion follows from Theorem \ref{theo:upper_bound_hj}.
\end{proof}

%
%

We can now fix the values of the coefficients $\tilde{r} \in (1,r]$, $\kappa>1$, and $\eta>1$ to be employed in the sequel, consistently with Corollary \ref{cor:a priori bound1} (if $s'<r$) and Corollary \ref{cor:a priori bound2} (if $s' \geq r$). As will appear later in the proofs of Lemma \ref{lem:intbyparts} and Proposition \ref{prop:existence_solution} , these coefficients must satisfy the following:
\begin{equation*}
\big[ s' \geq \tilde{r} \big], \quad \big[ \kappa \geq p \big], \quad \text{and} \quad
\big[ \text{$A$ is not constant $\Longrightarrow \tilde{r} \geq p$} \big].
\end{equation*}
This is the reason why four subcases have been introduced in Hypothesis (H5) and why we have a specific definition of the coefficients for each of the subcases. In order to deal with the case 2B, we need the following lemma.

\begin{lemma} \label{lemma:case2B}
Assume that $s' \geq r$ and $A$ is constant, that is, consider the case 2B of Assumption (H5).
Then the corresponding condition:
\begin{equation} \label{eq:conditionH5}
\Big[ s' \geq 1+d \Big] \quad
\text{or} \quad
\Big[ s' < 1+d \quad \text{and} \quad \frac{s'(1+d)}{d-s'+1} > p \Big]
\end{equation}
is satisfied if and only if there exists $\tilde{r} \in (1,r]$ such that
\begin{equation} \label{eq:existence_r}
\bar{\kappa} \Big( \tilde{r}, \min \big(p, {\textstyle \frac{s'}{\tilde{r}} } \big) \Big)
\geq p.
\end{equation}
\end{lemma}

\begin{proof} Several cases must be distinguished.
\begin{itemize}
\item Case (i): $s' > p$. In that case, either $s' \geq 1+d$ or $s' < 1+d$ and then
\begin{equation*}
\frac{s'(1+d)}{d-s'+1} > s' > p.
\end{equation*}
Thus, if $s' > p$, then \eqref{eq:conditionH5} holds true.
Then we can set $\tilde{r}= \frac{s'}{p} > 1$. We have $\min \big( p, {\textstyle \frac{s'}{\tilde{r}}} \big)= p$ and therefore
$\kappa = \bar{\kappa}(\tilde{r},p) \geq p$, by inequality \eqref{eq:low_bound_kappa}.
\item Case (ii): $s' \leq p$. Then whatever the choice of $\tilde{r} \in (1,r]$, we have
\begin{equation*}
\tilde{p}(\tilde{r}):= \min \big( p, {\textstyle \frac{s'}{\tilde{r}}} \big) = {\textstyle \frac{s'}{\tilde{r}}}.
\end{equation*}
\begin{itemize}
\item Case (iia): $s' > 1+d$. Then we can chose $\tilde{r}$ sufficiently close to 1, so that
\begin{equation*}
\frac{s'-d}{\tilde{r}} > 1.
\end{equation*}
Then we have $\tilde{p}(\tilde{r}) > 1 + \frac{d}{\tilde{r}}$, thus $\kappa= \bar{\kappa}(\tilde{r},\tilde{p}(\tilde{r}))= \infty$.
\item Case (iib): $s' \leq 1+d$. Then whatever the choice of $\tilde{r} \in (1,r]$, we have
$\tilde{p}(\tilde{r}) < 1 + \frac{d}{\tilde{r}}$ and therefore, condition \eqref{eq:existence_r} is equivalent to:
\begin{equation*}
\exists \tilde{r} \in (1,r], \quad
\frac{s'(d+1)}{d+\tilde{r}-s'} \geq p.
\end{equation*}
The above condition is clearly satisfied if and only if either $s'= 1+d$ or $s'<1+d$ and
$\frac{s'(d+1)}{d+1-s'} > p$.
\end{itemize}
\end{itemize}
\end{proof}

We can finally fix $\tilde{r}$, $\kappa$, and $\eta$.
\begin{itemize}
\item In cases 1A and 1B (i.e.\@ $s'<r$), we set
\begin{equation*}
\tilde{r}= s', \quad \kappa= \bar{\kappa}(\tilde{r},1), \quad \eta= \bar{\eta}(\tilde{r},1).
\end{equation*}
Then we have $\kappa \geq \tilde{r} \geq p$. In case 1A, $\tilde{r} \geq p$.
\item In case 2A (i.e.\@ $s' \geq r$ and $A$ is not constant), we set $\tilde{r}= p$.
In case 2B (i.e.\@ $s' \geq r$ and $A$ is constant), we assign a value to $\tilde{r}$ so that \eqref{eq:existence_r} holds true.
In both cases 2A and 2B, we set
\begin{equation*}
\kappa= \bar{\kappa} \Big( \tilde{r}, \min \big(p, {\textstyle \frac{s'}{\tilde{r}} } \big) \Big) \quad
\text{and} \quad
\eta= \bar{\eta} \Big( \tilde{r}, \min \big(p, {\textstyle \frac{s'}{\tilde{r}} } \big) \Big).
\end{equation*}
In case 2A, we have $\kappa \geq \tilde{r} = p$ by inequality \eqref{eq:low_bound_kappa}. In case 2B, we have $\kappa \geq p$ by definition.
\end{itemize}

\begin{remark}
In case 2B, it is easy to deduce from the proof of Lemma \ref{lemma:case2B} an explicit $\tilde{r} \in (1,r]$ such that \eqref{eq:existence_r} holds. Note that the obtained $\tilde{r}$ may not be the best one (i.e.\@ the largest one). For example, if $s' \geq pr$, then one can take $\tilde{r}= r$. Then $\frac{s'}{\tilde{r}} \geq p$ and therefore $\kappa = \bar{\kappa}(\tilde{r},p) \geq p$,
by inequality \eqref{eq:low_bound_kappa}.
\end{remark}

\subsection{The relaxed problem}

We propose in this subsection an appropriate relaxation of problem \eqref{eq:hjb_problem}.
Let $\mathcal{K}$ denote the set of triplets $(u,P,\gamma) \in L^\kappa(Q) \times L^{s'}(0,T) \times L^{p}(Q)$ such that $D u + \phi^\intercal P \in L^r(Q;\R^d)$, $Du \in L^{\tilde{r}}(Q;\R^d)$,
and such that \eqref{eq:subsolution_HJ} holds in the sense of distributions.

The following statement explains that $u$ has a ``trace" in a weak sense.

\begin{lemma}
	Let $f \in L^1(Q)$ and let $u \in L^1(Q)$ satisfy $D u \in L^{\tilde r}(Q;\R^d)$ and
	\begin{equation} \label{eq:inequality}
	-\partial_t u - A_{ij}\partial_{ij}u \leq f, \quad u(T) \leq u_T
	\end{equation}
	in the sense of distributions, i.e.~for every non-negative function $\vartheta \in C_c^\infty(\bb{T}^d \times \intoc{0,T})$ we have
	\begin{equation*}
	\iint_Q \del{u\partial_t \vartheta  + \partial_j(A_{ij}\vartheta)\partial_i u - f \vartheta} \dif x \dif t \leq \int_{\bb{T}^d} \vartheta(x,T)u_T(x)\dif x.
	\end{equation*}
	Then, for any $C^1$ map $\vartheta \colon [0,T] \times \bb{T}^d \rightarrow \bb{R}$, the function
	\begin{equation*}
	t \mapsto \int_{\bb{T}} \vartheta(x,t) u(x,t) dx
	\end{equation*}
	has a BV representative on $[0,T]$.
	In particular, for any nonnegative $C^1$ map $\vartheta \colon [0,T] \times \bb{T}^d \rightarrow \bb{R}$, one has the integration by parts formula: for any $0 \leq t_1 \leq t_2 \leq T$,
	\begin{equation} \label{eq:ibp}
	- \Big[ \int_{\bb{T}^d} \vartheta u \Big]_{t_1}^{t_2}
	+ \int_{t_1}^{t_2} \int_{\bb{T}^d} u \partial_t \vartheta
	+ \langle D \vartheta, A D u \rangle
	+ \vartheta \partial_i A_{ij} \partial_j u 
	\leq \int_{t_1}^{t_2} \int_{\bb{T}^d} f \vartheta.
	\end{equation}
\end{lemma}

\begin{proof}
	First, observe that $x \mapsto u(x,t)$ is a well-defined $L^1(\bb{T}^d)$ function for a.e.~$t \in (0,T)$.
	Then by standard convolution smoothing arguments, one can check that \eqref{eq:ibp} holds for a.e.~$t_1,t_2 \in [0,T]$ with $t_1 \leq t_2$.
	Indeed, if $\xi_\varepsilon$ is a convolution kernel, then $u_\varepsilon = \xi_\varepsilon \ast u, f_\varepsilon = \xi_\varepsilon \ast f$ can be shown to satisfy
	\begin{equation} 
	-\partial_t u_\varepsilon - A_{ij}\partial_{ij}u_\varepsilon \leq f_\varepsilon + R_\varepsilon
	\end{equation}
	where $R_\varepsilon \to 0$ in $L^1$ as $\varepsilon \to 0$.
	Then integration by parts implies, for $0 < t_1 \leq t_2 < T$ and $\varepsilon$ small enough, that
	\begin{equation} 
	- \Big[ \int_{\bb{T}^d} \vartheta u_\varepsilon \Big]_{t_1}^{t_2}
	+ \int_{t_1}^{t_2} \int_{\bb{T}^d} u_\varepsilon \partial_t \vartheta
	+ \langle D \vartheta, A D u_\varepsilon \rangle
	+ \vartheta \partial_i A_{ij} \partial_j u_\varepsilon 
	\leq \int_{t_1}^{t_2} \int_{\bb{T}^d} (f_\varepsilon + R_\varepsilon) \vartheta.
	\end{equation}
	Since $u_\varepsilon(\cdot,t) \to u(\cdot,t)$ in $L^1(\bb{T}^d)$ for a.e.~$t$, and likewise $u_\varepsilon \to u, D u_\varepsilon \to D u,$ and $f_\varepsilon \to f$ in $L^1$, so  by letting $\varepsilon \to 0$ we deduce the \eqref{eq:ibp} for a.e.~$t_1,t_2 \in [0,T]$ with $t_1 \leq t_2$.
	
	Now define, for a.e.~$t \in [0,T]$, the functions
	\begin{equation*}
	G(t) = \int_{\bb{T}^d} \vartheta(x,t)u(x,t)\dif x + F(t),
	\ \
	F(t) = \int_{t}^{T} \int_{\bb{T}^d} u \partial_t \vartheta
	+ \langle D \vartheta, A D u \rangle
	+ \vartheta \partial_i A_{ij} \partial_j u - f \vartheta.
	\end{equation*}
	Now $F$ is absolutely continuous, being the integral of an $L^1(0,T)$ function.
	By what we have shown $G(t)$ is increasing on its domain, and moreover $G(T) \leq \int_{\bb{T}^d} \vartheta(x,T)u_T(x)\dif x$ by hypothesis.
	Thus $I := G-F$ is BV, and \eqref{eq:ibp} indeed continues to hold for all $0 \leq t_1 \leq t_2 \leq T$, even if we replace $\int_{\bb{T}^d} \vartheta(x,t)u(x,t)\dif x$ by any value between $I(t+)$ and $I(t-)$.
\end{proof}

We extend the functional $D$ to triplets $(u,P,\gamma) \in \mathcal{K}$:
\begin{equation*}
D(u,P,\gamma) =
-\int_{\bb{T}^d} u(x,0^+) m_0(x) \dif x
+ \int_0^T \Phi^*\del{P(t)} \dif t
+ \iint_Q F^*\del{x,\gamma(x,t)} \dif x \dif t.
\end{equation*}
We consider the following relaxation of problem \eqref{eq:hjb_problem}:
\begin{equation} \label{eq:relaxed_pb_hj}
\inf_{(u,P,\gamma) \in \mathcal{K}} D(u,P,\gamma).
\end{equation}

\begin{proposition} \label{prop:justif_relaxation}
We have
\begin{equation*}
\inf_{(u,P,\gamma) \in \mathcal{K}_0} D(u,P,\gamma)
= \inf_{(u,P,\gamma) \in \mathcal{K}} D(u,P,\gamma).
\end{equation*}
\end{proposition}

The proof requires an integration by parts formula, established in the following lemma.

\begin{lemma}\label{lem:intbyparts}
Let $(u,P,\gamma)\in {\mathcal K}$ and $(m,w) \in {\mathcal K}_1$ satisfy \eqref{eq:bound_mw}. Assume that $mH^*(\cdot, -w/m)\in L^1(Q)$. Then
\begin{equation*}
\gamma m \in L^1((0,T) \times \bb{T}^d), \quad
\langle P(\cdot), { \textstyle \int_{\bb{T}^d} } \phi(x) w(x,\cdot) \dif x \rangle \in L^1(0,T)
\end{equation*}
and for almost all $t \in (0,T)$ we have
\begin{equation}
\label{ineq:ineqfirst-plan}
\int_{\bb{T}^d} (u(T) m_T -  u(t)m(t)) \dif x  + \int_t^{T}\int_{\bb{T}^d} \del{m \gamma  + mH^* \big(x, {\textstyle -\frac{w}{m}} \big) + \langle P(t), \phi(x) w(x,t) \rangle }\dif x \dif t \; \geq \; 0,
\end{equation}
and
\begin{equation}
\label{ineq:ineqfirst-plan_i}
		\int_{\bb{T}^d} (u(t)m(t) - u(0)m_0) \dif x + \int_0^{t}\int_{\bb{T}^d}\del{m \gamma  + mH^*\del{x,{\textstyle -\frac{w}{m}} } + \langle P(t), \phi(x,t)w(x,t) \rangle }\dif x \dif t \; \geq \; 0.
\end{equation}
Moreover, if equality holds in the inequality \eqref{ineq:ineqfirst-plan} for $t=0$, then $w= -mD_\xi H(\cdot,D u + \phi^\intercal P)$ a.e.
\begin{comment}
and, in the first order case, $-\partial_t u^{ac}(t,x) + H(x,D u(t,x)) = \alpha(t,x)$ for $m$-a.e.~$(t,x) \in (0,T) \times \bb{T}^d$, where $\partial_t u^{ac}$ is the absolutely continuous part of the measure $\partial_t u$.
\end{comment}
\end{lemma}
	
\begin{proof}
In the interest of smoothing \( (m,w) \) by convolution, extend the pair to \( [-1,T+1] \times \mathbb{T}^d \) by defining \( m = m_0 \) on \( [-1,0 ] \), \( m = m(T) \) on \( [T,T+1] \), and \( w(s,x) = 0 \) for \( (s,x) \in (-1,0) \cup (T,T+1) \times \mathbb{T}^d \). Let \( \tilde{A}_{ij} \) be an extension of \( A_{ij} \) with \( \tilde{A}_{ij} = A_{ij} \) if \( t \in (0,T) \) and zero otherwise. Note that with these described extensions, \( (m,w) \) solves 
    \begin{equation*}
    \partial_tm - \partial_{ij}(\tilde{A}_{ij}(t,x)m) + \nabla \cdot w = 0 \ \ \ \text{ on } \left( -1,T+1 \right) \times \mathbb{T}^d. 
    \end{equation*}
    Let \( \xi^\epsilon = \xi^\epsilon(t,x) = \xi^\epsilon_1(t)\xi^\epsilon_2(x) \) be a smooth convolution kernel with support in \( B_\epsilon \). We smoothen the pair \( (m,w) \) with this kernel in a standard way into \( (m_\epsilon, w_\epsilon) = (\xi^\epsilon \ast m,\xi^\epsilon \ast w) \). Then \( (m_\epsilon, w_\epsilon) \) solves
    \begin{equation}
        \label{eq:byparts_i}
         \partial_t m_\epsilon - \partial_{ij}(\tilde{A}_{ij} m_\epsilon) + \nabla \cdot w_\epsilon = \partial_j R_\epsilon \ \ \ \text{ in } \left(-\frac{1}{2}, T + \frac{1}{2} \right)
    \end{equation}
 in the sense of distributions, where
 \begin{equation}
     \label{eq:byparts_ii}
     R_\epsilon := [\xi^\epsilon, \partial_j \tilde{A}_{ij}](m) + [\xi^\epsilon, \tilde{A}_{ij}\partial_j](m).
 \end{equation}
 Here we use again the commutator notation \cite{diperna89}
 \begin{equation}
     \label{eq:commutator}
     [\xi^\epsilon, c](f) := \xi^\epsilon \star (cf) - c(\xi^\epsilon \star f).
      \end{equation}
    By \cite[Lemma II.1]{diperna89}, we have that \( R_\epsilon \to 0 \) in \( L^q \), since \( m \in L^q \) and \( \tilde{A}_{ij} \in W^{1,\infty}. \) Fix time \( t \in (0,T) \) at which \( u(t^+) = u(t^-) = u(t) \) in \( L^\kappa(\mathbb{T}^d) \) and \( m_\epsilon(t) \) converges to \( m(t) \). We have the following inequality based on the equality in \eqref{eq:HJ subsolution},    
    \begin{equation}
        \label{ineq:HJ subsolution}
        -\partial_t u - A_{ij} \partial_{ij}u + H(x, D u + \phi^\intercal P) \leq \gamma. 
    \end{equation}
    Integrating this inequality against \( m_\epsilon \) yields
    \begin{equation}
    \begin{split}
        \int_t^T \sint{ u \partial_t m_\epsilon + \partial_i u \partial_j( \tilde{A}_{ij} m_\epsilon ) + m_\epsilon H(x, D u + \phi^\intercal P)} + \sint{ m_\epsilon(t)u(t) &- m_\epsilon(T) u_T} \\
        &\leq \int_t^T \sint{ \gamma m_\epsilon}. 
        \end{split}
    \end{equation}
By \eqref{eq:byparts_i}, we have 
\begin{equation*}
    \int_t^T \sint u \partial_t m_\epsilon + \partial_i u \partial_j (\tilde{A}_{ij} m_\epsilon) = \int_t^T \sint{-\partial_i u R_\epsilon + \langle D u, w_\epsilon \rangle},
\end{equation*}
while the convexity of \( H \) in the last variable gives
\begin{equation}
\label{ineq:convexityH}
    \int_t^T \sint{ -m_\epsilon H^*\left( x, {\textstyle -\frac{w_\epsilon}{m_\epsilon} } \right) \leq \int_t^T \sint{ \langle w_\epsilon, D u  + \phi^\intercal P \rangle + m_\epsilon H(x, D u  + \phi^\intercal P).}}
\end{equation}
Combining these results yields 
\begin{equation*}
    \sint{ m_\epsilon(t)u(t)} \leq \sint{m_\epsilon(T) u_T} + \int_t^T \sint{ m_\epsilon \left( \gamma + H^* \left( x, {\textstyle -\frac{w_\epsilon}{m_\epsilon} } \right) \right)} +\langle \phi w_\epsilon ,  P \rangle  + \partial_j u R_\epsilon.
\end{equation*}
Following now \cite{cardaliaguet2015second}, we have that as \( D u \in L^{\tilde{r}} \) (where we recall that \( \tilde{r} \geq p \) or \( A \) is a constant matrix), and as \( m \) is continuous in time with respect to the topology on \( P(\mathbb{T}^d)\), we have, as \( \epsilon \to 0 \),
\begin{align*}
    \int_t^T \sint{ \partial_j u R_\epsilon} \rightarrow \ & 0, \\
    \int_t^T \sint{ -m_\epsilon H^* \left( x, {\textstyle -\frac{w_\epsilon}{m_\epsilon} } \right) }  \rightarrow \ &  \int_t^T \sint{ -m H^* \left( x, {\textstyle -\frac{w}{m} } \right) }, \\
    \sint{m_\epsilon(T)u_T} \rightarrow \ & \sint{m(T)u_T}.
\end{align*}
(For the second limit, cf.~\cite{cardaliaguet2014mean}.)
We also claim that
\begin{equation*}
	\int_t^T \int_{\bb{T}^d} \ip{\phi w_\epsilon}{P} \to \int_t^T \int_{\bb{T}^d} \ip{\phi w}{P}.
\end{equation*}
To see this, recall Equation \eqref{eq:bound_mw} from Lemma \ref{lemma:duality}.
If $\frac{1}{s} + \frac{1}{pr} \geq 1$, we deduce that $w \in L^s$ and therefore $w_\epsilon \to w$ in $L^s$; from this the claim follows immediately.
Otherwise, if $\frac{1}{s} + \frac{1}{pr} < 1$, then by Hypothesis (H1) we assume that $\phi$ is constant.
Therefore, we have
\begin{equation*}
	\int_t^T \int_{\bb{T}^d} \ip{\phi w_\epsilon}{P}
	= \int_t^T  \ip{\int_{\bb{T}^d} \phi w_\epsilon}{P}
	= \int_t^T  \ip{\xi^\epsilon_1 \ast \int_{\bb{T}^d} \phi w}{P}
	\to \int_t^T  \ip{\int_{\bb{T}^d} \phi w}{P}
\end{equation*}
because $t \mapsto \int_{\bb{T}^d} \phi w$ is in $L^s$.
Now since \( u \in L^\kappa(Q) \), \( m \in L^q(Q) \), and \( \kappa \geq p \), \( m_\epsilon u \) strongly converges to \( mu \) in \( L^1(Q) \) and thus up to a subsequence, \( \sint{m_\epsilon(t) u(t)} \to \sint{ m(t)u(t)} \) a.e. We now have that 
\begin{equation*}
    \sint{m(t) u(t)} \dif x \leq \sint{m(T) u_T} \dif x + \int_t^T \sint{ m\left( \gamma + H^* \left(x, -{\textstyle \frac{w}{m} } \right)  \right) + \langle P, \phi w \rangle \dif x \dif t}.
\end{equation*}
An analogous argument produces the other desired inequality, so now assume that equality holds in inequality \eqref{ineq:ineqfirst-plan_i} with \( t = 0 \). Then there is \( t^* \in (0,T) \) where equality holds with \( t = t^* \). Let
\[ E_\sigma(t) := \left\{ (s,y), s \in [t,T], m\left( H^*\left( y, {\textstyle -\frac{w}{m} } \right) + H(x, D u + \phi^\intercal P) \right) \geq  -\langle w, D u + \phi^\intercal P \rangle + \sigma \right\}. \]
If \( | E_\sigma (t)| > 0 \), then for \( \epsilon > 0 \) small enough, the set of \( s,y \) satisfying
\[ m_\epsilon \left(H^*\left( y , {\textstyle -\frac{w_\epsilon}{m_\epsilon} } \right)+ H(x, D u + \phi^\intercal P) \right) \geq -\langle w_\epsilon, D u + \phi^\intercal P \rangle + \frac{\sigma}{2} \]
has measure larger than \( \frac{|E_\sigma(t)|}{2} \). Then by \eqref{ineq:convexityH}, for the fixed choice of \( \epsilon \),
\[ \int_{t^*}^T \sint{ - m_\epsilon H^* \left(x, {\textstyle -\frac{w_\epsilon}{m_\epsilon} } \right)} \leq \int_{t^*}^T \sint{ \langle w_\epsilon, D u + \phi^\intercal P \rangle + m_\epsilon H(x, D u + \phi^\intercal P)} - |E_\sigma(t)| \sigma/ 4,  \]
whereby we obtain strict inequality in \eqref{ineq:ineqfirst-plan_i} with \( t = t^* \), a contradiction. Thus \( |E_\sigma(t)| = 0 \) for any \( \sigma \) and a.e. \( t\),
\[ \langle -w, D u + \phi^\intercal P \rangle = m \left( H(x, D u + \phi^\intercal P) + H^*(y,{\textstyle -\frac{w}{m} }) \right), \]
and hence 
\[ w = -mD_\xi H( \cdot, D u + \phi^\intercal P) \ \ a.e. \ \ \text{ in } \ \ (0,T) \times \mathbb{T}^d. \]    
\end{proof}

\begin{proof}[Proof of Proposition \ref{prop:justif_relaxation}]
It is clear that the value of the relaxed problem is smaller than the value of problem \eqref{eq:hjb_problem}. It remains to show the other inequality.
For any $(m,w) \in \mathcal{K}_1$ with $mH^*(- w/m) \in L^1(Q)$, we have, by Fenchel-Young inequality and Lemma \ref{lem:intbyparts},
\begin{align*}
D(u,P,\gamma)
\geq \ & - \int_{\bb{T}^d} u(0) m_0
+ \int_0^T \Big( \big\langle P(t) , {\textstyle \int_{\bb{T}^d}} \phi w \big\rangle - \Phi( {\textstyle \int_{\bb{T}^d}} \phi w) \Big)
+ \iint_Q \big( \gamma m - F(m) \big) \\
\geq \ & - \int_{\bb{T}^d} u_T m(T)
- \iint_Q m H^* \Big( - \frac{w}{m} \Big) - \int_0^T \Phi( {\textstyle \int_{\bb{T}^d}} \phi w) - \iint_Q F(m) \\
= \ & - B(m,w).
\end{align*}
Maximizing the right-hand side with respect to $(m,w)$, we obtain with Lemma \ref{lemma:duality} that
\begin{equation*}
D(u,P,\gamma) \geq - \inf_{(m,w) \in \mathcal{K}_1} B(m,w) = \inf_{(u,P,\gamma) \in \mathcal{K}_0} D(u,P,\gamma),
\end{equation*}
which concludes the proof.
\end{proof}

\subsection{Existence of a relaxed solution}

We establish now the existence of a relaxed solution.

\begin{proposition} \label{prop:existence_solution}
The relaxed problem \eqref{eq:relaxed_pb_hj} has at least one solution $(u,P,\gamma) \in \mathcal{K}$.
\end{proposition}

\begin{proof}
Let $(u_n,P_n,\gamma_n)$ be a minimizing sequence for problem \eqref{eq:hjb_problem}. By Proposition \eqref{prop:justif_relaxation}, it is also a minimizing sequence for the relaxed problem \eqref{eq:relaxed_pb_hj}.
We can, without loss of generality, assume that $\gamma_n \geq 0$, so long as we only require $u_n$ to be a viscosity solution to the Hamilton-Jacobi equation.
Let us replace $\gamma_n$ with its positive part, i.e.~$(\gamma_n)_+ := \max\{\gamma_n,0\}$.
Then we replace $u_n$ with $\tilde u_n$, the continuous viscosity solution of
\begin{equation*}
	-\partial_t \tilde u_n - A_{ij}\partial_{ij}\tilde u_n + H(D \tilde u_n + \phi^\intercal P_n) = (\gamma_n)_+, \quad \tilde u_n(x,T) = u_T(x).
\end{equation*}
	By \cite{ishii95}, $\tilde u_n$ also satisfies this equation in the sense of distributions, and thus the new triple $(\tilde u_n,P_n,(\gamma_n)_+)$ is also a member of $\s{K}$. We have $\tilde{u}_n \geq u_n$ and $F^*(\gamma_n)= F^*((\gamma_n)_+)$ for all $(x,t) \in Q$.
Therefore, $D(\tilde u_n,P_n,(\gamma_n)_+) \leq D(u_n,P_n,\gamma_n)$, and thus the new sequence also minimizes $D$.
The arguments below will then apply to $(\tilde u_n,P_n,(\gamma_n)_+)$ in place of $(u_n,P_n,\gamma_n)$.

\textbf{Step 1: [Bounds for $(\gamma_n)$, $(P_n)$, and $(D u_n)$]:}
\\
All constants $C$ used in this part of the proof are independent of $n$.
We integrate \eqref{eq:subsolution_HJ} against $m_0$ on $Q$ and obtain
\begin{equation} \label{eq:integrated_hj}
\int_{\bb{T}^d} u_n(0)m_0
+ \iint_Q \partial_j u_n \partial_i (A_{ij} m_0)
+ \iint_Q H(D u_n + \phi^\intercal P_n) m_0
\leq \iint_Q \gamma_n m_0 + \int_{\bb{T}^d} u_T m_0.
\end{equation}
Let us recall that $m_0 \geq \frac{1}{C}$.
The Hamiltonian term can be bounded from below as:
\begin{equation} \label{eq:bound_below_h}
\iint_Q H(D u_n + \phi^\intercal P_n) m_0 \geq 
\frac{1}{C} \| D u_n + \phi^\intercal P_n \|_r^r - C.
\end{equation}
In light of the regularity assumptions on $A$ and $m_0$, we also have that
\begin{equation} \label{eq:second_term}
\Big|
\iint_Q \partial_j u_n \partial_i(A_{ij} m_0)
\Big|
\leq C \| D u_n \|_1
\leq C \| D u_n \|_{\tilde{r}}.
\end{equation}
Finally, the right-hand side of \eqref{eq:integrated_hj} is bounded by $C \| \gamma_n \|_{p} + C$. Combining this estimate with \eqref{eq:bound_below_h} and \eqref{eq:second_term}, we obtain that
\begin{equation} \label{eq:existence1}
\int_{\bb{T}^d} u_n(0)m_0 + \frac{1}{CB} \| D u_n + \phi^\intercal P_n \|_r^r - C \| D u_n \|_{\tilde{r}} \leq C \| \gamma_n \|_p + C
\end{equation}
for any choice of \( B \geq 1 \). The constants $C$ used are also independent of $B$.
Now we use the fact that $(u_n,P_n,\gamma_n)$ is a minimizing sequence and the growth assumptions on $F^*$ and $\Phi^*$ to derive
\begin{equation} \label{eq:existence2}
- \int_{\bb{T}^d} u_n(0) m_0
+ \frac{1}{C} \| P_n \|_{s'}^{s'}
+ \frac{1}{C} \| \gamma_n \|_{p}^{p} - C
\leq D(u_n,P_n,\gamma_n) \leq C.
\end{equation}
Summing up \eqref{eq:existence1} and \eqref{eq:existence2}, we obtain
\begin{equation} \label{eq:prop_ex_bound}
\frac{1}{C B} \| D u_n + \phi^\intercal P_n \|_r^r - C \| D u_n \|_{\tilde{r}} + \frac{1}{C} \| P_n \|_{s'}^{s'} + \frac{1}{C} \| \gamma_n \|_p^p \leq C \| \gamma_n \|_p + C.
\end{equation}
Now by H\"older's inequality we have
\begin{equation*}
\| Du \|_{\tilde{r}}^{\tilde{r}}
\leq C \big( \| Du + \phi^\intercal P \|_{\tilde{r}}^{\tilde{r}} + \| \phi^\intercal P \|_{\tilde{r}}^{\tilde{r}} \big)
\leq C \big( \| Du + \phi^\intercal P \|_r^r + \| P \|_{s'}^{s'} + 1 \big)
\end{equation*}
and so
\begin{equation} \label{eq:prop_ex_bound2}
\frac{1}{CB} \| D u_n \|_{\tilde{r}}^{\tilde{r}}+ \left[ \frac{1}{C} - \frac{C}{B} \right] \| P_n \|_{s'}^{s'} - C \| D u_n \|_{\tilde{r}} + \frac{1}{C} \| \gamma_n \|_p^p \leq C \| \gamma_n \|_p + C.
\end{equation}
We fix now \( B= 2C^2. \)
The terms \( \| D u_n \|_{{\tilde{r}}} \), \( \| \gamma_n \|_p \) can be absorbed.
For instance, the former can be absorbed by $ \| D u \|^{\tilde{r}}_{\tilde{r}}$ insofar as for an arbitrarily small $\varepsilon > 0$, there exists $C>0$ (depending on $\varepsilon$) such that
\begin{equation} \label{eq:second_term2}
\| D u_n \|_{\tilde{r}} \leq \varepsilon \| D u_n \|_{\tilde{r}}^{\tilde{r}} + C.
\end{equation}\label{ineq:step2_1}
Taking \( \varepsilon \) small enough, we finally deduce from \eqref{eq:prop_ex_bound2} the estimate
\begin{equation}
\| D u_n \|_{\tilde{r}}^{\tilde{r}} +  \| P_n \|_{s'}^{s'} +  \| \gamma_n \|_p^p \leq C,
\end{equation} 
so that $(\gamma_n)_{n \in \mathbb{N}}$ is bounded in $L^{p}(Q)$, $(P_n)_{n \in \mathbb{N}}$ is bounded in $L^{s'}((0,T);\R^k)$ and $(D u_n)_{n \in \mathbb{N}}$ is bounded in $L^{\tilde{r}}(Q)$.
Inequality \eqref{eq:prop_ex_bound} further shows that $Du_n + \phi^\intercal P_n$ is bounded in $L^r(Q;\R^d)$. This implies that
\begin{equation*}
\| H(D u_n + \phi^\intercal P_n) \|_1 \leq C.
\end{equation*}
\\
\textbf{Step 2 [Bound of $u_n$ in $L^{\kappa}(Q)$]:}
\\
Now that we have estimates on $P_n$ in $L^{s'}$, $\gamma_n$ in $L^p$, $D u_n$ in $L^{\tilde r}$, and $H(D u_n + \phi^\intercal P_n)$ in $L^1$, we can apply Corollary \ref{cor:a priori bound1} in case $s' < r$ or Corollary \ref{cor:a priori bound2} in case $s' \geq r$ and obtain $\enVert{u_n}_\kappa \leq C$, where $\kappa$ is defined at the end of Section \ref{subsection:estimates_hj}.
\\\\
\textbf{Step 3 [Conclusion]:}
\\
The rest of the proof is very similar to the proof of \cite[Proposition 5.4]{cardaliaguet2015second}, we only give the main lines.
By passing to a subsequence, we assume without loss of generality that
\begin{equation*}
\begin{array}{lll}
u_n \rightharpoonup \bar{u} \text{ in } L^{\kappa}(Q), \quad &
D u_n \rightharpoonup D \bar{u} \text{ in } L^{\tilde{r}}(Q), \quad &
Du_n + \phi^\intercal P_n \rightharpoonup D \bar{u} + \phi^\top \bar{P} \text{ in } L^{r}(Q;\R^d), \\
\gamma_n \rightharpoonup \bar{\gamma} \text{ in } L^{p}(Q), &
P_n \rightharpoonup \bar{P} \text{ in } L^{s'}(0,T). & 
\end{array}
\end{equation*}
%
Since $H$ is convex, $(\bar{u},\bar{P},\bar{\gamma}) \in \mathcal{K}$. By weak lower semicontinuity arguments, we have
\begin{equation*}
\liminf_{n \to \infty} \iint_Q F^*(\gamma_n) + \int_0^T \Phi^*(P_n) \geq
\iint_Q F^*(\bar{\gamma}) + \int_0^T \Phi^*(\bar{P}).
\end{equation*}
Using exactly the same arguments as in \cite[Proposition 5.4 (Step 3)]{cardaliaguet2015second}, one can prove that
\begin{equation*}
\limsup_{n \to \infty} \int_{\bb{T}^d} u_n(0)m_0 \leq \int_{\bb{T}^d} \bar{u}(0) m_0,
\end{equation*}
which proves the optimality of $(\bar{u},\bar{P},\bar{\gamma})$.
\end{proof}

\section{Existence and uniqueness of a solution for the MFG system} \label{sec:existence}

We prove in this section the existence and uniqueness of a \emph{weak solution} to the MFG system \eqref{eq:MFGC_bis}.

\begin{definition} \label{def:weak_sol}
We say that a quadruplet $(u,P,m,w) \in L^\kappa(Q) \times L^{s'}(0,T) \times L^q(Q) \times L^{\frac{r'q}{r'+q-1}}(Q)$ is a weak solution if
\begin{itemize}
\item[(i)] The following integrability conditions hold: $D u \in L^{\tilde{r}}(Q)$ and $mH^*( \cdot, -m/w)) \in L^1(Q)$.
\item[(ii)] Equation \eqref{eq:MFGC_bis}-(i) holds in the sense of distributions,
\begin{equation*}
-\partial_t u - A_{ij} \partial_{ij} u + H(D u + \phi^{\intercal} P) \leq f(m), \quad
u(T) \leq u_T
\end{equation*}
\item[(iii)] Equation \eqref{eq:MFGC_bis}-(ii) holds in the sense of distributions,
\begin{equation*}
\partial_t m - \partial_{ij} \big( A_{ij} m \big) - \nabla \cdot w= 0, \quad
m(0)= m_0,
\end{equation*}
\item [(iv)] Equations \eqref{eq:MFGC_bis}-(iii)-(iv) hold almost everywhere,
\item [(v)] The following equality holds:
\begin{equation} \label{eq:complementarity}
\iint_Q \Big( m f(m) + m H^*(-w/m) + \langle P, \phi w \rangle \Big)  + \int_{\bb{T}^d} m(T) u_T - \int_{\bb{T}^d} m_0 u(0)= 0.
\end{equation}
\end{itemize}
\end{definition}

\begin{theorem} \label{theorem:existence_uniqueness}
There exists a weak solution $(u,P,m,w)$ to the MFG system \eqref{eq:MFGC_bis}. It is unique in the following sense: if $(u,P,m,w)$ and $(u',P',m',w')$ are two solutions, then $m=m'$, $w=w'$, $P=P'$ a.e.\@ and $u=u'$ in $\{ m > 0\}$. 
\end{theorem}

\begin{theorem} \label{theorem:equi_mfg_ocp}
Let $(\bar{m},\bar{w}) \in \mathcal{K}_1$ be a minimizer of \eqref{eq:fp_problem} and $(\bar{u},\bar{P},\bar{\gamma})$ be a minimizer of \eqref{eq:relaxed_pb_hj}. Then, $(\bar{u},\bar{P},\bar{m},\bar{w})$ is a weak solution of the MFG system and $\bar{\gamma}= f(\bar{m})$.

Conversly, any weak solution $(\bar{u},\bar{P},\bar{m},\bar{w})$ of the MFG system is such that $(\bar{m},\bar{w})$ is the solution to \eqref{eq:fp_problem} and $(\bar{u},\bar{P},f(\bar{m}))$ is a solution to \eqref{eq:relaxed_pb_hj}.
\end{theorem}

\begin{proof}
\textbf{Part 1.}
Let $(\bar{m},\bar{w}) \in \mathcal{K}_1$ be the solution to \eqref{eq:fp_problem} and $(\bar{u},\bar{P},\bar{\gamma}) \in \mathcal{K}$ be a solution to \eqref{eq:relaxed_pb_hj}.
Conditions $(ii)$ and $(iii)$ of Definition \eqref{def:weak_sol} are already verified.
By Lemma \ref{lemma:duality} and Proposition \ref{prop:justif_relaxation}, these two problems have the same value, thus
\begin{align*}
0= \ & D(\bar{u},\bar{P},\bar{\gamma}) + B(\bar{m},\bar{w}) \\
= \ & \iint_Q \big( F^*(\bar{\gamma}) + F(\bar{m}) \big)
+ \int_0^T \Big( \Phi^*(\bar{P}) + \Phi \big( {\textstyle \int_{\bb{T}^d}} \phi \bar{w} \big) \Big) \\
& \qquad + \iint_Q \bar{m} H^*(-\bar{w}/\bar{m})
+ \int_{\bb{T}^d} u_T \bar{m}(T) - \int_{\bb{T}^d} \bar{u}(0) m_0.
\end{align*}
By the Fenchel-Young inequality, we have
\begin{align}
F^*(\bar{\gamma}) + F(\bar{m}) \geq \bar{\gamma} \bar{m} \quad & \text{for a.e.\@ $(x,t) \in Q$}, \label{eq:theo_equi_mfg_a} \\
\Phi^*(\bar{P}) + \Phi \big( {\textstyle \int_{\bb{T}^d}} \phi \bar{w} \big) \geq \big\langle \bar{P}, {\textstyle \int_{\bb{T}^d} } \phi \bar{w} \big\rangle \quad & \text{for a.e.\@ $t \in (0,T)$} \label{eq:theo_equi_mfg_b}
\end{align}
thus
\begin{equation} \label{eq:theo_equi_mfg}
0 \geq \iint_Q \Big( \bar{m} \bar{\gamma} + \bar{m} H^*(-\bar{w}/\bar{m}) + \langle P,\phi w \rangle \Big)+ \int_{\bb{T}^d} \bar{m}(T) u_T - \int_{\bb{T}^d} m_0 \bar{u}(0).
\end{equation}
This implies first that $\bar{m} H^*(-\bar{w}/\bar{m}) \in L^1(Q)$.
Moreover, by Lemma \ref{lem:intbyparts}, inequality \eqref{eq:theo_equi_mfg} is in fact an equality and $\bar{w}= -\bar{m} D_\xi H(D \bar{u} + \phi^{\intercal} \bar{P})$. Moreover, the equality holds a.e.\@ in \eqref{eq:theo_equi_mfg_a} and \eqref{eq:theo_equi_mfg_b} therefore, 
\begin{align*}
\bar{\gamma}= F'(\bar{m}) = f(\bar{m}) \quad & \text{for a.e.\@ $(x,t) \in Q$}, \\
\bar{P}= D \Phi \big( {\textstyle \int_{\bb{T}^d} } \phi \bar{w} \big) = \Psi \big( {\textstyle \int_{\bb{T}^d} } \phi\bar{w} \big) \quad & \text{for a.e.\@ $t \in (0,T)$}.
\end{align*}
Since \eqref{eq:theo_equi_mfg} is an equality and $\bar{\gamma} = f(\bar{m})$, \eqref{eq:complementarity} holds true. We conclude that $(\bar{u},\bar{P},\bar{m},\bar{w})$ is a weak solution to the MFG system.

\textbf{Part 2.}
Let $(\bar{u},\bar{P},\bar{m},\bar{w})$ be a weak solution to \eqref{eq:MFGC_bis}. Let $\bar{\gamma}= f(\bar{m})$. The growth condition on $f$ implies that $\bar{\gamma} \in L^{p}(Q)$. Therefore, $(\bar{m},\bar{w}) \in \mathcal{K}_1$ and $(\bar{u},\bar{P},\bar{\gamma}) \in \mathcal{K}$. It remains to show that $(\bar{u},\bar{P},\bar{\gamma})$ solves \eqref{eq:relaxed_pb_hj} and that $(\bar{m},\bar{w})$ solves \eqref{eq:relaxed_pb_hj}.

The argument is very similar to the one used in Proposition \ref{prop:justif_relaxation}. It mainly consists in showing that $D(\bar{u},\bar{P},\bar{\gamma}) + B(\bar{m},\bar{w})=0$.
Since $\bar{\gamma}= f(\bar{m})= F'(\bar{m})$ a.e., we have by convexity of $F$ that
\begin{equation*}
F(\bar{m}) + F^*(\bar{\gamma})= \bar{\gamma} \bar{m}, \quad \text{for a.e.\@ $(x,t) \in Q$.}
\end{equation*}
Similarly, since $\bar{P}= \Psi \big( {\textstyle \int_{\bb{T}^d} } \phi w \big)= D \Phi \big(  {\textstyle \int_{\bb{T}^d} } \phi w \big)$, we have
\begin{equation*}
\Phi \big( {\textstyle \int_{\bb{T}^d}} \phi w \big) + \Phi^*(\bar{P})
= \big\langle P, {\textstyle \int_{\bb{T}^d}} \phi w \big\rangle, \quad \text{for a.e.\@ $t \in (0,T)$.}
\end{equation*}
These two equalities and \eqref{eq:complementarity} yield:
\begin{align*}
D(\bar{u},\bar{P},\bar{\gamma}) + B(\bar{m},\bar{w}) = \ &
\iint_Q \big( F^*(\bar{\gamma}) + F(\bar{m}) \big)
+ \int_0^T \Big( \Phi \big( {\textstyle \int_{\bb{T}^d} } \phi \bar{w} \big) + \Phi^*(\bar{P}) \Big) \\
& \qquad \qquad + \iint_Q \bar{m} H^*(-\bar{w}/\bar{m}) + \int_{\bb{T}^d} \big( u_T \bar{m}(T) - \bar{u}(0) m_0 \big) \\
= \ & \iint_Q \bar{m} \bar{\gamma} + \iint_Q \langle \bar{P}, \phi \bar{w} \rangle + \int_{\bb{T}^d} u_T \bar{m}(T) - \bar{u}(0) m_0 + \iint_Q \bar{m} H^*(-\bar{w}/\bar{m}) \\
= \ & 0.
\end{align*}
As a consequence, we obtain
\begin{equation*}
\inf_{(u,P,\gamma) \in \mathcal{K}} D(u,P,\gamma) \leq D(\bar{u},\bar{P},\bar{\gamma})
= - B(\bar{m},\bar{w}) \leq - \min_{(m,w) \in \mathcal{K}_1} B(m,w).
\end{equation*}
The first and the last term being equal, the two above inequalities are equalities, which shows the optimality of optimality of $(\bar{u},\bar{P},\bar{\gamma})$ and $(\bar{m},\bar{w})$, respectively.
\end{proof}

\begin{proof}[Proof of Theorem \ref{theorem:existence_uniqueness}]
By Lemma \ref{lemma:duality}, problem \eqref{eq:fp_problem} has a solution $(\bar{m},\bar{w})$ and by Proposition \ref{prop:existence_solution}, problem \eqref{eq:relaxed_pb_hj} has a solution $(\bar{u},\bar{P},\bar{\gamma})$. By Theorem \ref{theorem:equi_mfg_ocp}, $(\bar{u},\bar{P},\bar{m},\bar{w})$ is a weak solution to the MFG system.

Now, let $(u_1,P_1,m_1,w_1)$ and $(u_2,P_2,m_2,w_2)$ be two weak solutions. By Theorem \ref{theorem:equi_mfg_ocp}, $(m_1,w_1)$ and $(m_2,w_2)$ are solutions to problem \eqref{eq:fp_problem}, they are therefore equal. Relation \eqref{eq:MFGC_bis}-(iii) implies that $P_1=P_2$. Let $(\bar{m},\bar{w},\bar{P})= (m_1,w_1,P_1)$ denote the common values. Let $\bar{\gamma}= f(\bar{m})$. Then $(u_1,\bar{P},\bar{\gamma})$ and $(u_2,\bar{P},\bar{\gamma})$ lie in $\mathcal{K}$ (by definition of weak solutions).

For $t \in (0,T)$, $(m,w) \in \mathcal{K}_1$, and $(u,P,\gamma) \in \mathcal{K}$, we introduce
\begin{align*}
B_t(m,w)= \ & \int_t^T \int_{\bb{T}^d} \Big( mH^*(-w/m) + F(m) \Big) + \int_t^T \Phi \big( {\textstyle \int_{\bb{T}^d} } \phi w \big) + \int_{\bb{T}^d} u_T m(T) \\
D_t(u,P,\gamma)= \ & - \int_{\bb{T}^d} u(0)m_0 + \int_t^T \Phi^*(P) + \int_t^T \int_{\bb{T}^d} F^*(\gamma).
\end{align*}
Proceeding as in the proof of Proposition \ref{prop:justif_relaxation}, we obtain that
\begin{equation*}
\inf_{(u,P,\gamma) \in \mathcal{K}} D_t(u,P,\gamma) \geq -B_t(\bar{m},\bar{w}).
\end{equation*}
By Lemma \ref{lem:intbyparts} and relation \eqref{eq:complementarity},
\begin{equation*}
\int_t^T \int_{\bb{T}^d} \Big( \bar{m} f( \bar{m}) + \bar{m} H^*(-\bar{w}/ \bar{m}) + \langle \bar{P}, \phi \bar{w} \rangle \Big)  + \int_{\bb{T}^d} \bar{m}(T) u_T - \int_{\bb{T}^d} \bar{m}(t) u_i(t)= 0
\end{equation*}
for a.e.\@ $t \in (0,T)$ and for $i=1,2$. Proceeding as in the proof of Theorem \ref{theorem:equi_mfg_ocp}, we obtain that $-B_t(\bar{m},\bar{w}) = D_t(u_i,\bar{P},\bar{\gamma})$. Thus $(u_1,\bar{P},\bar{\gamma})$ and $(u_2,\bar{P},\bar{\gamma})$ minimize $D_t$ over $\mathcal{K}$.

Let $\bar{u}= u_1 \vee u_2$. Adapting the proof in \cite[Theorem 6.2]{cardaliaguet2015second}, we obtain that $(\bar{u},\bar{P},\bar{\gamma}) \in \mathcal{K}$. Since $D_t(\bar{u}) \leq D_t(u_i)$, we deduce that $(\bar{u},P,\gamma)$ also minimize $D_t$. It follows that
\begin{equation*}
\int_{\bb{T}^d} u_1(t) \bar{m}(t)
= \int_{\bb{T}^d} u_2(t) \bar{m}(t)
= \int_{\bb{T}^d} \bar{u}(t) \bar{m}(t)
\end{equation*}
As $u_1 \leq \bar{u}$ and $u_2 \leq \bar{u}$, this implies that $u_1=u_2= \bar{u}$ a.e.\@ in $\{ \bar{m} > 0 \}$ and concludes the proof.
\end{proof}
	
	\section{Regularity estimates} \label{sec:regularity}
	
	In this section we adapt the methods used in \cite{graber2018sobolev,graber2019planning} to show that weak solutions of \eqref{eq:MFGC} possess extra regularity--Sobolev estimates in both space and time--not required by Definition \ref{def:weak_sol}.
	These estimates hold under general \emph{strong monotonicity} assumptions on the coupling $f(x,m)$ and \emph{coercivity} on the Hamiltonian.
	We divide our results into ``space regularity,'' i.e.~estimates on derivatives with respect to $x$, and ``time regularity,'' estimates on derivatives with respect to $t$.

	\subsection{Space regularity}
	Before stating the result, let us enumerate a few additional assumptions.
	
	\begin{assumption} \label{as:A const}
		$A_{ij}$ is constant.
	\end{assumption}
	\begin{assumption}[Strong monotonicity] \label{as:regularizing f}
		We have a Lipschitz estimate on $f$ of the form
		\begin{equation}
		\label{f Lipschitz in x}
		|f(x,m) - f(y,m)| \leq C(m^{q-1}+1)|x-y|\  \ \forall x,y \in \bb{T}^d, \ m \geq 0.
		\end{equation}
		We also assume that $f(x,m)$ is strongly monotone in $m$, that is, there exists $c_f > 0$ such that
		\begin{equation} \label{f strongly monotone}
		\del{f(x,\tilde m) - f(x,m)}(\tilde m - m) \geq c_f\min\{\tilde m^{q-2},m^{q-2}\}|\tilde m-m|^2 \ \forall \tilde m, m \geq 0, \ \tilde m \neq m.
		\end{equation}
		If $q < 2$ one should interpret $0^{q-2}$ as $+\infty$ in \eqref{f strongly monotone}.
		In this way, when $\tilde m = 0$, for instance, \eqref{f strongly monotone} reduces to $f(x,m)m \geq c_f m^q$, as in the more regular case $q \geq 2$.
	\end{assumption}
	\begin{assumption}[Coercivity] \label{as:coercivity}
		There exist $j_1,j_2:\bb{R}^d \to \bb{R}^d$ and $c_H > 0$ such that
		\begin{equation} \label{eq:Hcoercivity}
		H(x,\xi) + H^*(x,\zeta) - \xi \cdot \zeta \geq c_H|j_1(\xi) - j_2(\zeta)|^2.
		\end{equation}
		In particular, and in light of our restriction on the growth of $H$, we specify that $j_1(\xi) \sim |\xi|^{r/2-1}\xi$ and $j_2(\zeta) \sim |\zeta|^{r'/2-1}\zeta$.
	\end{assumption}
	\begin{assumption} \label{as:2nd x derivatives}
		$m_0 \in W^{2,\infty}(\bb{T}^d), u_T \in W^{2,\infty}(\bb{T}^d), \phi \in W^{2,\infty}\del{\bb{T}^d;\s{L}(\bb{R}^k,\bb{R}^d)}$, and $H^*$ is twice continuously differentiable in $x$ with	
		\begin{equation} \label{hyp:D_x^2 H}
		|D_{xx}^2 H^*(x,\zeta)| \leq C_H\del{|\zeta|^{r'} + 1}.
		\end{equation}
	\end{assumption}
	Notice that Assumption \ref{as:regularizing f} holds for the canonical case $f(x,m) = m^{q-1}$ or even if $f(x,m) = \tilde f(x)m^{q-1}$ for some strictly positive, Lipschitz continuous function $\tilde f$ on $\bb{T}^d$.
	Assumption \ref{as:coercivity} likewise holds for a canonical structure $H(x,\xi) = c(x)\abs{\xi}^r$ for some strictly positive, $\s{C}^2$ smooth function $c(x)$ on $\bb{T}^d$.
	
	\begin{proposition} \label{prop:space-regularity}
		Let Assumptions \ref{as:A const}, \ref{as:coercivity}, \ref{as:regularizing f}, \ref{as:2nd x derivatives} hold.
		Then, if $(u,m)$  is a weak solution of \eqref{eq:MFGC},
		$$\|m^{\frac{q}{2} - 1}D m\|_{L^2(Q} \leq C \hspace{0.3cm} \mbox{and } \; \; \|m^{1/2}D (j_1(D u))\|_{L^2(Q} \leq C,$$
		where $C$ is a constant depending only on the data.
	\end{proposition}
	
	Throughout we use the notation $g^\delta(x) = g(x+\delta)$ for any function depending on $x \in \bb{T}^d$.
	
	Take a smooth minimizing sequence $(u_n,P_n,\gamma_n) \in \s{K}_0$ for the dual problem.
	Integrate \eqref{eq:FP equation} by parts against \( u_n \) and rearrange to get
	\begin{equation} \label{eq:u_n-test-m}
	\iint_Q mH(x,D u_n + \phi^\intercal P_n) \dif x \dif t = \int_{\bb{T}^d} (u_Tm(T) - u_n(0) m_0 ) \dif x + \iint_Q \gamma_n m - \langle D u_n, w \rangle \dif x \dif t.
	\end{equation}
	\emph{Step 1.} The following estimates show that (up to a subsequence) $D u_n \rightharpoonup D u$ in $L^{\tilde r}_m([0,T] \times \bb{T}^d;\bb{R}^d)$ (see Section \ref{sec:opt ctrl hj} for definition of $\tilde r$, and NB $\tilde r \leq \min\{r,s'\}$):
	\\\\
	Using Young's inequality and Assumption (H2) we get
	\begin{multline} \label{eq:mgrad un 1}
	\frac{1}{C}\iint_Q \abs{D u_n + \phi^\intercal P_n}^{\tilde r} m \dif x\dif t \leq	
	\frac{1}{C}\iint_Q \abs{D u_n + \phi^\intercal P_n}^r m \dif x\dif t + C\\
	\leq \|u_T\|_\infty + \int_{\bb{T}^d}|u_n(0)|m_0
	+ \iint_Q \del{(\gamma_n)_+ m + Cm\abs{\frac{w}{m}}^{r'}} \dif x\dif t
	+ C.
	\end{multline}
	By possibly increasing \( C \) we get
	\begin{multline} \label{eq:mgrad un 2}
	\frac{1}{C}\iint_Q \abs{D u_n}^{\tilde r} m \dif x\dif t \\
	\leq \|u_T\|_\infty + \int_{\bb{T}^d}|u_n(0)|m_0
	+ \iint_Q \del{(\gamma_n)_+ m + Cm\abs{\frac{w}{m}}^{r'} + C\abs{\phi^\intercal P_n}^{\tilde r} m}\dif x\dif t + C\\
	\leq \|u_T\|_\infty + \int_{\bb{T}^d}|u_n(0)|m_0
	+ \iint_Q \del{(\gamma_n)_+ m + Cm\abs{\frac{w}{m}}^{r'}}\dif x\dif t
	+ C\int_0^T  | P_n |^{s'} \dif t + C.
	\end{multline}
	Since $P_n$ is bounded in $L^{s'}$,  we have that \( D u_n \) is bounded in \( L_m^{\tilde{r}} \) where we recall that \( \tilde{r} = \min(r,s') \). Thus, up to a subsequence, \( D u_n \rightharpoonup \zeta \) for some \( \zeta \in L_m^{\tilde r} \). The argument that \( \zeta = D u \) \( m-\)a.e.~follows as in \cite{graber2018sobolev}.
	We also have, up to a subsequence, that $P_n \rightharpoonup P$ in $L^{s'}(0,T)$, and thus also that $D u_n + \phi^\intercal P_n \rightharpoonup D u + \phi^\intercal P$ in $L^r_m([0,T] \times \bb{T}^d;\bb{R}^d)$. Indeed, the upper bound given by \eqref{eq:mgrad un 1} shows that \( D u_n + \phi^\intercal P_n \) converges weakly in \( L_m^r \), and its limit must be equal to $D u + \phi^\intercal P$ a.e.~by taking the limit of each summand. 
	
	\emph{Step 2.} Now use $u_n^\delta$ and $u_n^{-\delta}$ as test functions in \eqref{eq:FP equation} to get
	\begin{equation} \label{eq:u_n delta-test-m}
	\int_{\bb{T}^d} \del{u_T^\delta m(T) - u_n^\delta(0)m_0}\dif x = \iint_Q \del{H(x+\delta,D u_n^\delta + (\phi^\delta)^\intercal P_n)m-\gamma_n^\delta m + D u_n^\delta \cdot w} \dif x\dif t
	\end{equation}
	and
	\begin{equation} \label{eq:u_n -delta-test-m}
	\int_{\bb{T}^d} \del{u_T^{-\delta} m(T) - u_n^{-\delta}(0)m_0}\dif x = \iint_Q \del{H(x-\delta,D u_n^{-\delta} + (\phi^{-\delta})^T P_n)m-\gamma_n^{-\delta} m + D u_n^{-\delta} \cdot w} \dif x\dif t
	\end{equation}
	We have the optimality condition
	\begin{equation} \label{eq:optimality}
	\int_{\bb{T}^d} \del{u_T m(T) - u(0)m_0}\dif x = -\iint_Q \del{H^*\del{x,-\frac{w}{m}}m + P \cdot (\phi w) + f(x,m)m}  \dif x\dif t.
	\end{equation}
	Take $\eqref{eq:u_n delta-test-m} + \eqref{eq:u_n -delta-test-m} -2\eqref{eq:optimality}$ to get
	\begin{multline} \label{eq:space reg 1}
	\int_{\bb{T}^d} \del{\del{u_T^\delta+u_T^{-\delta}-2u_T} m(T) - u_n(0)\del{m_0^\delta+m_0^{-\delta}-2m_0}}\dif x\\
	= \iint_Q \del{H(x+\delta,D u_n^\delta + (\phi^\delta)^T P_n)m + H^*\del{x+\delta,-\frac{w}{m}}m  + \del{D u_n^\delta + (\phi^\delta)^TP_n} \cdot w} \dif x\dif t\\
	+\iint_Q \del{H(x-\delta,D u_n^{-\delta} + (\phi^{-\delta})^T P_n)m + H^*\del{x-\delta,-\frac{w}{m}}m  + \del{D u_n^{-\delta} + (\phi^{-\delta})^TP_n} \cdot w} \dif x\dif t\\
	+\iint_Q \del{\del{2f(x,m) -\gamma_n^\delta-\gamma_n^{-\delta}} m + 2P \cdot (\phi w) - P_n \cdot (\phi^\delta w + \phi^{-\delta} w)} \dif x\dif t - I
	\end{multline}
	where
	\begin{equation}
	I := \iint_Q \del{H^*\del{x+\delta,-\frac{w}{m}}+H^*\del{x-\delta,-\frac{w}{m}}-2H^*\del{x,-\frac{w}{m}}}m
	\end{equation}
	and where we have used
	\begin{equation}
	\int_{\bb{T}^d}  \del{u_n^\delta(0)+u_n^{-\delta}(0)-2u(0)}m_0\dif x
	= \int_{\bb{T}^d}  u_n(0)\del{m_0^\delta+m_0^{-\delta}-2m_0}\dif x.
	\end{equation}	
	Since $H$ is convex in the third argument, by the result of Step 1 and weak lower semicontinuity we have
	\begin{equation*}
	\iint_Q H(x\pm\delta,D u^{\pm \delta} + (\phi^{\pm \delta})^T P)m  \dif x\dif t
	\leq \liminf \iint_Q H(x\pm\delta,D u_n^{\pm \delta} + (\phi^{\pm \delta})^T P_n)m  \dif x\dif t.
	\end{equation*}
	Letting $n \to \infty$ in \eqref{eq:space reg 1} we obtain
	\begin{multline} \label{eq:space reg 2}
	\iint_Q \del{H(x+\delta,D u^\delta + (\phi^\delta)^T P)m + H^*\del{x+\delta,-\frac{w}{m}}m  + \del{D u^\delta + (\phi^\delta)^TP} \cdot w} \dif x\dif t\\
	+\iint_Q \del{H(x-\delta,D u^{-\delta} + (\phi^{-\delta})^T P)m + H^*\del{x-\delta,-\frac{w}{m}}m  + \del{D u^{-\delta} + (\phi^{-\delta})^TP} \cdot w} \dif x\dif t\\
	\leq
	\int_{\bb{T}^d} \del{\del{u_T^\delta+u_T^{-\delta}-2u_T} m(T) - u(0)\del{m_0^\delta+m_0^{-\delta}-2m_0}}\dif x + I\\
	+\iint_Q \del{\del{f(x+\delta,m^\delta)+f(x-\delta,m^{-\delta})-2f(x,m)} m + P \cdot \del{\del{\phi^\delta + \phi^{-\delta} -2\phi} w}} \dif x\dif t.
	\end{multline}
	By \cite[computation (4.25)]{graber2018sobolev} we have
	\begin{multline} \label{eq:f monotonicity estimates}
	\int_{\bb{T}^d} \del{f(x+\delta,m^\delta)+f(x-\delta,m^{-\delta})-2f(x,m)}m \dif x
	\\
	\leq
	C|\delta|^2 \left(1+\int_{\bb{T}^d} \min\{ m^\delta,m\}^{q} \dif x\right)
	- \frac{c_f}{2}\int_{\bb{T}^d} \min\{ (m^\delta)^{q-2}, m^{q-2}\}| m^\delta - m|^2 \dif x.
	\end{multline}
	Using estimate \eqref{eq:Hcoercivity} on the left-hand side of \eqref{eq:space reg 2}, then using $\abs{a+b}^2 \leq 2\abs{a}^2 + 2\abs{b}^2$, and combining with \eqref{eq:f monotonicity estimates}, then using Assumption \ref{as:2nd x derivatives} we deduce
	\begin{multline} \label{eq:space reg 3}
	\frac{c_H}{2}\iint_Q \abs{\del{D u^\delta + (\phi^\delta)^T P}^{r/2} - \del{D u^{-\delta} + (\phi^{-\delta})^T P}^{r/2}}m \dif x\dif t\\
	+\frac{c_f}{2}\int_{\bb{T}^d} \min\{ (m^\delta)^{q-2}, m^{q-2}\}\abs{m^\delta - m}^2 \dif x\\
	\leq
	\abs{\delta}^2\del{\enVert{u_T}_{W^{2,\infty}} + \enVert{m_0}_{W^{2,\infty}}\int_{\bb{T}^d} \abs{u(0)}\dif x + C_H\del{\enVert{\frac{w}{m}}_{L^r_m}+1}}\\
	\abs{\delta}^2\del{C \left(1+\int_{\bb{T}^d} \min\{ m^\delta,m\}^{q} \dif x\right)
	+\enVert{\phi}_{W^{2,\infty}}\iint_Q \abs{P \cdot  w} \dif x\dif t}.
	\end{multline}
	
	\subsection{Time regularity}
	
	As in the previous subsection, we enumerate our assumptions before stating the main result.
	\begin{assumption}
		\label{as:extra assumptions}
		We assume that $A_{ij} = 0$.
	\end{assumption}
	We remark that Assumption \ref{as:extra assumptions} is much stronger than Assumption \ref{as:A const} but appears to be necessary, for technical reasons that appear in the estimates below.

	\begin{assumption}[Strong monotonicity in time] \label{as:regularizing f in time}
	We assume that \eqref{f strongly monotone} holds.
		
		We assume that $\Psi$ is invertible, with inverse denoted
		 by $\del{\Psi^{-1}}(t,\cdot)$ (for instance, it suffices to assume that its primitive $\Phi$ is strictly convex).
		We assume that, for some constant $c_\Psi > 0$,
		\begin{equation} \label{Psi strongly monotone}
		\del{\Psi^{-1}(\tilde P) - \Psi^{-1}(P)}(\tilde P - P) \geq c_\Psi\min\cbr{\abs{\tilde P}^{s'-2},\abs{P}^{s'-2}}|\tilde P-P|^2 \ \ \forall t,\tau \in [0,T], \ P \in \bb{R}^k
		\end{equation}
	\end{assumption}

	\begin{proposition}
		Under Assumptions \ref{as:extra assumptions}, \ref{as:regularizing f in time}, and \eqref{hyp:D_x^2 H}, for every $\varepsilon > 0$, there exists a constant $C(\varepsilon)$, depending only on $\varepsilon$ and the data, such that
		\begin{equation}
		\enVert{\partial_t \del{m^{q/2}}}_{L^2\del{Q_\varepsilon}} + \enVert{\od{}{t}\del{P^{s'/2}}}_{L^2(\varepsilon,T-\varepsilon)} \leq C(\varepsilon)
		\end{equation}
		where $Q_\varepsilon := \bb{T}^d \times (\varepsilon,T-\varepsilon)$.
	\end{proposition}

	\begin{remark}
		The proposition could also be proved for data depending on time, in particular with $f(x,m)$ and $H(x,\xi)$ replaced by $f(t,x,m)$ and $H(t,x,\xi)$, respectively.
		The only additional assumption needed would be a Lipschitz estimate in $t$, where the Lipschitz constant can depend on $x$ (but not on $m$ or $\xi$).
	\end{remark}
	
	\begin{proof}
		Let $\varepsilon\in\bb{R}$ be small and $\eta:[0,T]\to[0,1]$ be smooth and compactly supported on $(0,T)$ such that $|\varepsilon|<\min\left\{{\rm{dist}}(0,{\rm{spt}}(\eta));{\rm{dist}}(T,{\rm{spt}}(\eta))\right\}$ and $\max_{t} \abs{\varepsilon\eta'(t)} < 1$.
	If $\varepsilon > 0$ we set $\eta_\varepsilon(t) = t + \varepsilon \eta(t)$, which is a strictly increasing bijection from $[0,T]$ to itself.
	Then we set $\eta_{-\varepsilon} = \eta_\varepsilon^{-1}$, which is also smooth by the inverse function theorem.
	For competitors $(u,P,\gamma)$ of the minimization problem for $\s{A}$, let us define
	\begin{equation*}
	u^\varepsilon(x,t):=u(x,\eta_{\varepsilon}(t));\ \ P^\varepsilon(t) = P(\eta_{\varepsilon}(t)); \ \ \gamma^\varepsilon(x,t):=\eta_{\varepsilon}'(t)\gamma(x,\eta_{\varepsilon}(t)).
	\end{equation*}
	Notice that by construction, if $t\in\{0,T\}$ then $u(x,t)=u^\varepsilon(x,t)$ and $\gamma(x,t)=\gamma^\varepsilon(x,t)$, provided that $\gamma(x,t)$ is well-defined. 
	
	Similarly, for competitors $(m,w)$  of minimization problem for $\s{B}$, we define
	$$m^\varepsilon(x,t):=m(x,\eta_{\varepsilon}(t));\ \ w^\varepsilon(x,t):=\eta_{\varepsilon}'(t)w(x,\eta_{\varepsilon}(t))$$
	and here as well if $t\in\{0,T\}$ then $m(x,t)=m^\varepsilon(x,t)$ and $w(x,t)=w^\varepsilon(x,t).$
	
	We define moreover perturbations on the data as 
	$$
	 \Phi^\varepsilon(t,v) := \eta_{\varepsilon}'(t)\Phi\del{v/\eta_{\varepsilon}'(t)},$$
	$$f^\varepsilon(t,x,m):=\eta_{\varepsilon}'(t) f(x,m);\ \ F^\varepsilon(t,x,m):= \eta_{\varepsilon}'(t) F(x,m),$$
	from which the Legendre transforms w.r.t.~the last variable satisfy
	$$(\Phi^\varepsilon)^*(t,P) = \eta_{\varepsilon}'(t)\Phi(P), \
	(F^\varepsilon)^*(t,x,\gamma):= \eta_{\varepsilon}'(t) F^*(x,\gamma/\eta_{\varepsilon}'(t)).$$
	Finally, we define
	$$H^\varepsilon(t,x,\xi):= \eta_{\varepsilon}'(t) H(x,\xi), \ \ {\rm{thus}}\ (H^\varepsilon)^*(x,\zeta):= \eta_{\varepsilon}'(t) H^*(x,\zeta/\eta_{\varepsilon}'(t)).$$
	
	\emph{Step 1.}
	Take a smooth minimizing sequence $(u_n,P_n,\gamma_n)$ in $\s{K}_0$.
	Use $u_n^{\pm \varepsilon}$ as a test function in $\partial_t m + \nabla \cdot w = 0$ to get
	\begin{equation} \label{eq:time reg1}
	\int_{\bb{T}^d}\del{u_Tm(T)-u_n(0)m_0}\dif x
	\geq
	\iint_Q\del{H^{ \varepsilon}(t,x,D u_n^{ \varepsilon} + \phi^\intercal P^\varepsilon(t))m-\gamma_n^{ \varepsilon}m + D u_n^\varepsilon \cdot w}\dif x\dif t
	\end{equation}
	and
	\begin{equation} \label{eq:time reg2}
	\int_{\bb{T}^d}\del{u_Tm(T)-u_n(0)m_0}\dif x
	\geq
	\iint_Q\del{H^{- \varepsilon}(t,x,D u_n^{- \varepsilon} + \phi^\intercal P^{-\varepsilon}(t))m-\gamma_n^{- \varepsilon}m + D u_n^{-\varepsilon} \cdot w}\dif x\dif t.
	\end{equation}
	Take $\eqref{eq:time reg1}+\eqref{eq:time reg2}-2\eqref{eq:optimality}$ to get
	\begin{multline}
	\int_{\bb{T}^d} 2\del{u(0) - u_n(0)}m_0 \dif x\\
	\geq
	\iint_Q\del{H^{ \varepsilon}(x,D u_n^{\varepsilon} + \phi^\intercal P^{\varepsilon}) + H^*\del{x,-\frac{w}{m}} + \del{D u_n^{\varepsilon} + \phi^\intercal P^{\varepsilon}} \cdot \frac{w}{m}}m \dif x \dif t\\
	+ \iint_Q\del{H^{- \varepsilon}(x,D u_n^{- \varepsilon} + \phi^\intercal P^{-\varepsilon}) + H^*\del{x,-\frac{w}{m}}
	 + \del{D u_n^{-\varepsilon}+\phi^\intercal P^{-\varepsilon}} \cdot \frac{w}{m}}m \dif x \dif t\\
	+ \iint_Q \del{2f(x,t,m) - \gamma_n^\varepsilon - \gamma_n^{-\varepsilon}}m \dif x \dif t
	+ \iint_Q \del{2P \cdot (\phi w) - P^\varepsilon \cdot (\phi w) - P^{-\varepsilon} \cdot (\phi w)} \dif x \dif t.
	\end{multline}
	Letting $n \to \infty$ we get
	\begin{multline} \label{eq:time reg estimate}
	\iint_Q\del{H^{ \varepsilon}(x,D u^{\varepsilon} + \phi^\intercal P^{\varepsilon}) + (H^*)^\varepsilon\del{x,-\frac{w}{m}} + \del{D u^{\varepsilon} + \phi^\intercal P^{\varepsilon}} \cdot \frac{w}{m}}m \dif x \dif t\\
	+ \iint_Q\del{H^{- \varepsilon}(x,D u^{- \varepsilon} + \phi^\intercal P^{-\varepsilon}) + (H^*)^{-\varepsilon}\del{x,-\frac{w}{m}}
		+ \del{D u^{-\varepsilon}+\phi^\intercal P^{-\varepsilon}} \cdot \frac{w}{m}}m \dif x \dif t\\
	+ \iint_Q \del{2f(x,m) - f^\varepsilon(t,x,m^\varepsilon) - f^{-\varepsilon}(t,x,m^{-\varepsilon})}m \dif x \dif t\\
	+ \iint_Q \del{2P \cdot (\phi w) - P^\varepsilon \cdot (\phi w) - P^{-\varepsilon} \cdot (\phi w)} \dif x \dif t
	\leq R(\varepsilon)
	\end{multline}
	where
	\begin{equation} \label{eq:R}
	R(\varepsilon) := \iint_Q\del{(H^*)^\varepsilon\del{x,-\frac{w}{m}} + (H^*)^{-\varepsilon}\del{x,-\frac{w}{m}}
	- 2H^*\del{x,-\frac{w}{m}}}\dif x \dif t.
	\end{equation}
	Arguing as in \cite[Proposition 3.3, Step 1]{graber2019planning} and using the estimate on $D_{xx}^2 H^*$, we have $R(\varepsilon) = O(\varepsilon^2)$.
	
	Next we perform changes of variables and the relation $P = D_v\Phi(t,\int \phi w)$ to rewrite
	\begin{multline}
	\iint_Q \del{2P \cdot (\phi w) - P^\varepsilon \cdot (\phi w) - P^{-\varepsilon} \cdot (\phi w)} \dif x \dif t
	=\iint_Q \del{P^\varepsilon - P} \cdot (\phi w^\varepsilon-\phi w) \dif x \dif t\\
	=\int_0^T\del{P^\varepsilon - P} \cdot \del{\del{\Psi^{-1}}^\varepsilon(t,P^\varepsilon) - \Psi^{-1}(t,P)} \dif t.
	\end{multline}
	Using the same argument as in \cite[Proposition 3.3, Step 4]{graber2019planning}, Assumption \ref{as:regularizing f in time} implies
	\begin{multline}\label{eq:P coercivity}
	\int_0^T\del{P^\varepsilon - P} \cdot \del{\del{\Psi^{-1}}^\varepsilon(t,P^\varepsilon) - \del{\Psi^{-1}}(t,P)} \dif t\\
	\geq \frac{c_\Psi}{2}\int_0^T \min\cbr{\abs{P^\varepsilon(t)},\abs{P(t)}}^{s'-2}\abs{P^\varepsilon(t)-P(t)}^2\dif t -C\abs{\varepsilon}^2\int_0^T \abs{P(t)}^{s'}\dif t.
	\end{multline}
	We use an analogous argument (or see \cite[Proposition 3.3, Step 4]{graber2019planning}) we deduce
	\begin{multline} \label{eq:f coercivity}
	\iint_Q \del{2f(x,m) - f^\varepsilon(t,x,m^\varepsilon) - f^{-\varepsilon}(t,x,m^{-\varepsilon})}m \dif x \dif t\\
	=\iint_Q \del{f^\varepsilon(t,x,m^\varepsilon)-f(x,m)}\del{m^\varepsilon-m}\dif x \dif t\\
	\geq \frac{c_f}{2}\iint_Q \min\cbr{m^\varepsilon,m}^{q-2}\abs{m^\varepsilon-m}^2 \dif x \dif t
	- C\abs{\varepsilon}^2\iint_Q m^{q}\dif x \dif t.
	\end{multline}
	Finally, using Assumption \ref{as:coercivity} we get
	\begin{multline} \label{eq:H coercivity time}
	\iint_Q\del{H^{ \varepsilon}(x,D u^{\varepsilon} + \phi^\intercal P^{\varepsilon}) + (H^*)^\varepsilon\del{x,-\frac{w}{m}} + \del{D u^{\varepsilon} + \phi^\intercal P^{\varepsilon}} \cdot \frac{w}{m}}m \dif x \dif t\\
	+ \iint_Q\del{H^{- \varepsilon}(x,D u^{- \varepsilon} + \phi^\intercal P^{-\varepsilon}) + (H^*)^{-\varepsilon}\del{x,-\frac{w}{m}}
		+ \del{D u^{-\varepsilon}+\phi^\intercal P^{-\varepsilon}} \cdot \frac{w}{m}}m \dif x \dif t\\
	\geq c_H\iint_Q\abs{j_1\del{D u^{\varepsilon} + \phi^\intercal P^{\varepsilon}} - j_2\del{\frac{w}{m}}}^2m \dif x \dif t
	+ c_H\iint_Q\abs{j_1\del{D u^{-\varepsilon} + \phi^\intercal P^{-\varepsilon}} - j_2\del{\frac{w}{m}}}^2m \dif x \dif t\\
	\geq \frac{c_H}{2}\iint_Q\abs{j_1\del{D u^{\varepsilon} + \phi^\intercal P^{\varepsilon}}-j_1\del{D u^{-\varepsilon} + \phi^\intercal P^{-\varepsilon}}}^2m \dif x \dif t
	\end{multline}
	Combining \eqref{eq:R}, \eqref{eq:P coercivity}, \eqref{eq:f coercivity}, and \eqref{eq:H coercivity time} with \eqref{eq:time reg estimate}, we get
	\begin{multline} \label{eq:final time regularity}
	\frac{c_H}{2}\iint_Q\abs{j_1\del{D u^{\varepsilon} + \phi^\intercal P^{\varepsilon}}-j_1\del{D u^{-\varepsilon} + \phi^\intercal P^{-\varepsilon}}}^2m \dif x \dif t \\
	+ \frac{c_f}{2}\iint_Q \min\cbr{m^\varepsilon,m}^{q-2}\abs{m^\varepsilon-m}^2 \dif x \dif t\\
	+ \frac{c_\Psi}{2}\int_0^T \min\cbr{\abs{P^\varepsilon(t)},\abs{P(t)}}^{s'-2}\abs{P^\varepsilon(t)-P(t)}^2\dif t
	\leq C\abs{\varepsilon}^2,
	\end{multline}
	where we have used the estimates on $\iint_Q m^q \dif x \dif t$ and $\int_0^T \abs{P(t)}^{s'}\dif t$.
	The conclusion follows.
	\end{proof}
	
	\bibliographystyle{alpha}
	\bibliography{mybib}
\end{document}